\documentclass[]{article}
\usepackage[T2A]{fontenc}
\usepackage[cp1251]{inputenc}
\usepackage{amssymb}
\usepackage{amsmath}
\usepackage{amsthm}
\usepackage{hyperref}

\title{\bf Congruences on Menger algebras}
\author{Wieslaw A. Dudek and Valentin S. Trokhimenko}

\begin{document}
\sloppy \maketitle

\newtheorem{theorem}{Theorem}[section]
\newtheorem{proposition}[theorem]{Proposition}
\newtheorem{definition}[theorem]{Definition}
\newtheorem{collolary}[theorem]{Corollary}

\begin{abstract}\noindent
We discuss some types of congruences on Menger algebras of rank $n$, which are generalizations of the principal left and right
congruences on semigroups.  We also study congruences  admitting various types of  cancellations
and describe their relationship with strong subsets.
\end{abstract}

\section{Preliminaries}

Let $A$ be a nonempty set. Consider the family $\Phi$ of mappings
$f\colon A^n\rightarrow A$ closed under the $(n+1)$-ary
composition $f[g_1\ldots g_n]$ defined by:
\[
f[g_1\ldots g_n](x_1,\ldots,x_n)=f(g_1(x_1,\ldots,x_n),\ldots,
g_n(x_1,\ldots,x_n)).
\]
Such defined an $(n+1)$-ary composition is called the \textit{Menger superposition} of $n$-place functions (cf. \cite{Dudtro3, SchTr}) and satisfies the so-called \textit{superassociative law}:
\begin{equation} \label{e1}
f[g_1\ldots g_n][h_1\ldots h_n]=f[g_1[h_1\ldots h_n]\ldots
g_n[h_1\ldots h_n]],
\end{equation}
where $f,g_i,h_i\in\Phi$, $i=1,\ldots,n$.

Obtained algebra $(\Phi, \mathcal{O})$, where $\mathcal{O}\colon (f,g_1,\ldots,g_n)\mapsto f[g_1\ldots g_n]$ is a \textit{Menger algebra of $n$-place functions}. Their abstract analog $(G,o)$, i.e., a nonempty set $G$ with an $(n+1)$-ary operation $o:(f,g_1,\ldots,g_n)\mapsto f[g_1\ldots g_n]$ satisfying \eqref{e1} is called a {\it Menger algebra $(G,o)$ of rank $n$}.

Let $(G,o)$ be a Menger algebra of rank $n$, \ $e_1,\ldots,e_n$ -- pairwise different elements that not belong to the set $G$, $\overline{e}=(e_1,\ldots,e_n)$ and $B=G^n\cup\{\overline{e}\}$. In addition, we assume that $g[e_1\ldots e_n]=g$, $e_i[g_1\ldots
g_n]=g_i$ and $e_i[e_1\ldots e_n]=e_i$ for all $g,g_1,\ldots,g_n\in G$ and each $i=1,\ldots,n$.
Elements $e_1,\ldots,e_n$ corresponds to the projectors $I_1,\ldots,I_n$ defined by $I_i(x_1,\ldots,x_n)=x_i$.

Since the $(n+1)$-ary operation $o$ is superassociative, the set $B$ with a binary operation:
$$
\overline{x}*\overline{y}=(x_1,\ldots,x_n)*(y_1,\ldots,y_n)=(x_1[y_1\ldots
y_n],\ldots, x_n[y_1\ldots y_n])
$$
is a semigroup with the identity $\overline{e}$. Obviously $(G^n,*)$ is a subsemigroup of $(B,*)$.

\medskip

Let us consider the set $T_{n}(G)$ of all expressions, called
\textit{polynomials}, in the alphabet $G\cup\{\,[ \; ],x\}$, where the square brackets and $x$ do not
belong to $G$, defined as follows:
\begin{itemize}
\item[$(a)$] $x\in T_n(G)$,
\item[$(b)$] if $\,i\in \{1,\ldots ,n\}, \ \ a,b_{1},\ldots ,b_{i-1},b_{i+1},\ldots
,b_{n}\in G$, \ $t\in T_{n}(G)$, then \ $a[b_{1}\ldots
b_{i-1}t\,b_{i+1},\ldots b_{n}]\in T_{n}(G)$,
\item[$(c)$] $T_{n}(G)$ contains precisely those polynomials which are
defined according to $(a)$ and $(b)$.
\end{itemize}

Every polynomial $t\in T_{n}(G)$ defines on $(G,o)$ an \textit{elementary translation}
$t:x\mapsto t(x)$. The polynomial and the elementary translation
defined by it will be noted by the same letter. Obviously, $t_1\circ t_2\in T_n(G)$
 for any $t_1,t_2\in T_n(G)$, where $(t_1\circ t_2)(x)=t_1(t_2(x))$.
 The identity map on the set $G$ belongs to $T_n(G)$ also.

\medskip

A subset $H$ of a Menger algebra $(G,o)$ of rank $n$ is called
\begin{itemize}
\item \textit{a normal $v$-complex } if for all $ g_1, g_2\in G$,
$t\in T_n(G)$
$$g_1\in H\wedge g_2\in H\wedge t(g_1)\in H\longrightarrow t(g_2)\in H,
$$
\item \textit{a normal $l$-complex } if for all $g_1,g_2\in G$ and
$\overline{x}\in B$
\begin{equation}\label{e2}
  g_1\in H\wedge g_2\in H\wedge g_1[\overline{x}]\in H\longrightarrow
  g_2[\overline{x}]\in H,
\end{equation}
\item \textit{a normal bicomplex } if for any
$g_1,g_2\in G$, $t\in T_n(G)$, $\overline{x}\in B$
\begin{equation}\label{e3}
  g_1\in H\wedge g_2\in H\wedge t(g_1[\overline{x}])\in H\longrightarrow
  t(g_2[\overline{x}])\in H,
\end{equation}
\item \textit{an $l$-ideal } if for all $x,h_1,\ldots,h_n\in G$
$$
(h_1,\ldots,h_n)\in G^n\setminus (G\setminus H)^n\longrightarrow
x[h_1\ldots h_n]\in H,
$$
\item \textit{an $i$-ideal } ($1\leqslant i\leqslant n$), if for all
$h,u\in G$, $\overline{w}\in G^n$
$$
h\in H\longrightarrow u[\overline{w}|_ih]\in H,
$$
where $u[\overline{w}|_ih]$ denotes $u[w_1\ldots
w_{i-1}\,h\,w_{i +1}\ldots w_n]$,
\item \textit{an $s$-ideal } if for all $h,x_1,\ldots,x_n\in G$
$$
h\in H\longrightarrow h[x_1\ldots x_n]\in H,
$$
\item \textit{an $sl$-ideal } if $H$ is both $s$-ideal and $l$-ideal.
\end{itemize}
It is clear that $H$ is an $l$-ideal if and only if it is an
$i$-ideal for all $i=1,\ldots,n$.

\medskip

A binary relation $\rho$ defined on a Menger algebra $(G,o)$ of rank $n$ is called
\begin{itemize}
\item {\it stable } if for all $x,y,x_i,y_i\in G$, \ $i=1,\ldots,n$
$$
(x,y),(x_1,y_1),\ldots,(x_n,y_n)\in\rho\longrightarrow (x[x_1\ldots x_n],
y[y_1\ldots y_n])\in\rho ,
$$
\item {\it $l$-regular } if for any $x,y,z_i\in G$, \ $i=1,\ldots,n$
$$
(x,y)\in\rho\longrightarrow (x[z_1\ldots z_n], y[z_1\ldots z_n])\in\rho ,
$$
\item {\it $v$-regular } if for all $x_i,y_i,z\in G$, \ $i=1,\ldots,n$
$$
(x_1,y_1),\ldots,(x_n,y_n)\in\rho\longrightarrow (z[x_1\ldots x_n],
z[y_1\ldots y_n])\in\rho ,
$$
\item {\it $i$-regular } if for any $u,x,y\in G$, \ $\overline{w}\in G^n$
$$
(x,y)\in\rho\longrightarrow (u[\overline{w}|_ix], u[\overline{w}|_iy])\in\rho ,
$$
\item {\it $l$-cancellative } if for all $x,y\in G$, \ $(z_1,\ldots,z_n)\in G^n$
$$
(x[z_1\ldots z_n], y[z_1\ldots z_n])\in\rho\longrightarrow (x,y)\in\rho,
$$
\item {\it $v$-cancellative } if for all $x,y,u\in G$, \ $\overline{w}\in G^n$,
$i=1,\ldots,n$
$$
(u[\overline{w}|_ix], u[\overline{w}|_iy])\in\rho\longrightarrow (x,y)\in\rho ,
$$
\item {\it $lv$-cancellative } if it is both $l$-cancellative and $v$-cancellative.
\end{itemize}

A reflexive and transitive binary relation is $v$-regular if and only if it is
$i$-regular for all $i=1,\ldots,n$. A binary relation is stable if and only if it is
$l$-regular and $v$-regular (cf. \cite{Dudtro3}). A stable equivalence is a
\textit{congruence}. By a \textit{$v$-congruence} (\textit{$l$-congruence}) we mean an equivalence relation which is $v$-regular ($l$-regular, respectively). It is easy to see that a relation $\rho\subset G\times G$ is $v$-cancellative if and only if
\begin{equation}\label{e4}
(t(x),t(y))\in\rho\longrightarrow (x,y)\in\rho
\end{equation}
holds for all $x,y\in G$ and each translation $t\in T_n(G)$.

The above terminology is inspired by the terminology used in semigroups.
A Menger algebra of rank $n=1$ is a semigroup,  so, in this case, an $l$-regular relation is left regular, and a $v$-regular (vector regular) relation is right regular in the classical sense.

In this paper we describe principal $v$-congruences and principal
$l$-congruences on a Menger algebra of rank $n$. Obtained results
are similar. The main scheme of proofs also is similar, but the
concept of strong subsets (used in the proofs) is different in any
case. Thus, these proofs are not analogous.

In the last section we characterize bistrong subsets of Menger algebras and their
connections with congruences.

\section{Principal $v$-congruences on Menger algebras}

Let $(G,o)$ be a Menger algebra of rank $n$, $H$ its nonempty subset. Put
\[
\rho_{_H} = \{(a,t)\in G\times T_n(G)\,|\,t(a)\in H\},
\]
\[
\rho_{_H}\langle a\rangle =\{t\in T_n(G)\,|\,t(a)\in H\},
\]
\[\rho^{\circ}_{_H}\langle t\rangle =\{a\in G\,|\,t(a)\in H\},
\]
\[
\mathcal{R}_H=\{(a,b)\in G\times G \,|\,\rho_{_H}\langle
a\rangle  = \rho_{_H}\langle b\rangle\},
\]
\[
\mathcal{R}^*_H = \{(a,b)\in G\times G\,|\,\rho_{_H}\langle
a\rangle =\rho_{_H}\langle b\rangle\neq\varnothing\} .
\]

The relation $\mathcal{R}_H $ is an equivalence on $(G,o)$ but the relation $\mathcal{R}^*_H$ is only symmetric and transitive. The set $W_H=\{a\in G\,|\,\rho_{_H}\langle a\rangle=\varnothing\}$ is a complement of
the domain of $\mathcal{R}^*_H$. Thus, if $W_H\neq\varnothing$, then $W_H$ is the
$\mathcal{R}_H $-class. Obviously, $\mathcal{R}^*_H =\mathcal{R}_H\cap(G\!\setminus\!W_H)\!\times\! (G\!\setminus\!W_H)$.

It is not difficult to see that for any $a,b\in G$ we have
\begin{eqnarray}&\label{e5}
(a,b)\in\mathcal{R}_H\longleftrightarrow\Big(\forall
t\in T_n(G)\Big)\Big(t(a)\in H\longleftrightarrow t(b)\in H\Big), &\\[4pt]
&\label{e6} a\in W_H\longleftrightarrow\Big(\forall t\in
T_n(G)\Big)\Big(t(a)\not\in H\Big). &
\end{eqnarray}
Observe that $H\cap W_H=\varnothing$. Indeed, for $h\in H\cap W_H$ we have $h\in H$ and
$t(h)\not\in H$ for any $t\in T_n(G)$. This for the identity translation $t(x)=x$ gives $h\in H$ and $h\not\in H$, which is impossible. So, $H\cap W_H=\varnothing$.

\begin{proposition}\label{P2.1}
$\mathcal{R}_H$ is a $v$-congruence.
\end{proposition}
\begin{proof} Let $(a,b)\in\mathcal{R}_H$. Then, according to (\ref{e5}), for every $t\in
T_n(G)$ we have $t(a)\in H\longleftrightarrow t(b)\in H$. In particular $t(u[\overline{w}\,|_ia])\in H\longleftrightarrow t(u[\overline{w}\,|_ib])\in H$ for every $u\in G$, $\overline{w}\in G^n$ and $i\in\{1,\ldots,n\}$. So, $(u[\overline{w}\,|_ia], u[\overline{w}\,|_ib])\in\mathcal{R}_H$. Hence, the
relation $\mathcal{R}_H$ is $i$-regular for every $i=1,\ldots,n$. Consequently, it is
a $v$-congruence on $(G,o)$.
\end{proof}

\begin{proposition}\label{P2.2}
$W_H$ is an $l$-ideal, provided that it is a nonempty set.
\end{proposition}
\begin{proof}
Let $W_H\neq\varnothing$ and $a\in W_H$. Then, according to (\ref{e6}), $t(a)\not\in H$ for each translation $t\in T_n(G)$. Hence $t_1(u[\overline{w}\,|_ia])\not\in H$ for $t_1(u[\overline{w}\,|_ia])\not\in H$, where
$t_1\in T_n(G)$, $u\in G$, $\overline{w}\in G^n$ and $i\in\{1,\ldots,n\}$. Thus,
$u[\overline{w}\,|_ia]\in W_H$ for each $i=1,\ldots,n$, i.e., $W_H$
is an $l$-ideal of a Menger algebra $(G,o)$.
\end{proof}

The relation $\mathcal{R}_H $ is called the \textit{principal $v$-congruence},
the relation  $\mathcal{R}^*_H$ -- the \textit{principal partial $v$-congruence}
on a Menger algebra $(G,o)$. The set $W_H$ is called the \textit{$v$-residue of $H.$}

\begin{proposition}\label{P2.3}
If a nonempty subset $H$ of a Menger algebra $(G,o)$ is its normal $v$-complex, then it is an
$\mathcal{R}_H$-class different from $W_H$.
\end{proposition}
\begin{proof}
Let $h_1,h_2\in H$. Suppose that $t(h_1)\in H$ for some $t\in T_n(G)$. Then $h_1,h_2,t(h_1)\in H$, which gives $t(h_2)\in H$ because $H$ is a normal $v$-complex. So, $t(h_1)\in H$ implies $t(h_2)\in H$. Analogously, $t(h_2)\in H$ implies $t(h_1)\in H$. Thus, according to (\ref{e5}), $(h_1,h_2)\in\mathcal{R}_H$ for all $h_1,h_2\in H$. Hence $H$ is contained in some $\mathcal{R}_H$-class. Denote this class by $X$.  Now let $h\in H$ and $g$ be an arbitrary element of $X$. Then $(h,g)\in\mathcal{R}_H$, i.e., for each $t\in T_n(G)$ we have $t(h)\in H\longleftrightarrow t(g)\in H$. This for the identity translation gives $h\in H\longleftrightarrow g\in H$. So, $H=X$. Since $H\cap W_H=\varnothing$, the $\mathcal{R}_H$-class $H$ is different from $W_H$.
\end{proof}

\begin{proposition}\label{P2.4}
Any $v$-congruence $\varepsilon$ on a Menger algebra $(G,o)$ has the form
$$
\varepsilon = \bigcap\{\mathcal{R}_H\,|\,H\in G/\varepsilon\},
$$
where $G/\varepsilon$ is the set of all $\varepsilon $-classes.
\end{proposition}
\begin{proof}
Let $(g_1,g_2)\in\varepsilon$. Consider an arbitrary $\varepsilon$-class $H.$ If $t(g_1)\in H$ for some $t\in
T_n(G)$, then $(t(g_1),t(g_2))\in\varepsilon$ because the relation $\varepsilon$ is $i$-regular
for every $i=1,\ldots,n$. So $t(g_2)\in H$. Analogously, $t(g_2)\in H$ gives $t(g_1)\in H.$
Thus, according to (\ref{e5}), we have $(g_1,g_2)\in\mathcal{R}_H,$ i.e., $(g_1,g_2)\in\varepsilon$ implies $(g_1,g_2)\in\mathcal{R}_H$. Hence $\varepsilon\subseteq\bigcap\{\mathcal{R}_H\,|\,H\in G/\varepsilon\}$.

Conversely, let $(g_1,g_2)\in\bigcap\{\mathcal{R}_H\,|\,H\in G/\varepsilon\}$. Then, according to (\ref{e5}), for all $H\in G/\varepsilon$ and $t\in T_n(G)$ we have
$$
t(g_1)\in H\longleftrightarrow t(g_2)\in H .
$$
From this, for the identity translation $t$ and the $\varepsilon$-class $H_1$ containing $g_1$, we obtain $g_1\in
H_1\longleftrightarrow g_2\in H_1$. Thus $(g_1,g_2)\in\varepsilon$. Consequently, $\bigcap\{\mathcal{R}_H\,|\,H\in G/\varepsilon\}\subseteq\varepsilon$. This completes the proof.
\end{proof}

\begin{definition}\rm
A nonempty subset $H$ of a Menger algebra $(G,o)$ is called \textit{$v$-admissible} if it is an equivalence class
of some $v$-congruence on $(G,o)$.
\end{definition}

\begin{proposition}\label{P2.6}
Each $v$-admissible subset $H$ of a Menger algebra $(G,o)$ is contained in some
$\mathcal{R}_H$-class. Moreover, if $\,\rho_{_H}\langle h\rangle \neq\varnothing$ for
some $h\in H,$ then this $\mathcal{R}_H$-class is different from $W_H$.
\end{proposition}
\begin{proof}
Let $H$ be a $v$-admissible subset of a Menger algebra $(G,o)$. Then it is some $\varepsilon$-class of a
$v$-congruence $\varepsilon$. Therefore, by Proposition \ref{P2.4}, we obtain
$\varepsilon\subseteq\mathcal{R}_H$. So, $H$ is contained in some $\mathcal{R}_H$-class. If
$\rho_{_H}\langle h\rangle\neq\varnothing$ for some $h\in H,$ then $h\not\in W_H$. This means that the
$\mathcal{R}_H$-class containing $H$ cannot be the set $W_H$.
\end{proof}

\begin{definition}\rm
A subset $H$ of a Menger algebra $(G,o)$ is called \textit{strong} if
\begin{equation}\label{e7}
  \rho_{_H}\langle a\rangle\cap\rho_{_H}\langle b\rangle\neq\varnothing\longrightarrow
  \rho_{_H}\langle a\rangle=\rho_{_H}\langle b\rangle
\end{equation}
for all $a,b\in G$.
\end{definition}

\begin{theorem}\label{T2.8}
For a subset $H$ of a Menger algebra $(G,o)$ the following
conditions are equivalent:
\begin{itemize}
  \item [$(a)$] \ $H$ is strong.
  \item [$(b)$] \ For any $x,y\in G$ and $t_1,t_2\in T_n(G)$
\begin{equation}\label{e8}
t_1(x)\in H\wedge t_1(y)\in H\wedge t_2(y)\in H\longrightarrow t_2(x)\in H .
\end{equation}
  \item [$(c)$] \ For any $t_1,t_2\in T_n(G)$
\begin{equation}\label{e9}
  \rho^{\circ}_{_H}\langle t_1\rangle\cap\rho^{\circ}_{_H}\langle
t_2\rangle\neq\varnothing\longrightarrow\rho^{\circ}_{_H}\langle
t_1\rangle  =\rho^{\circ}_{_H}\langle t_2\rangle .
\end{equation}
\end{itemize}
\end{theorem}
\begin{proof}
$(a)\longrightarrow (b)$. Let $H$ be a strong subset of a Menger algebra $(G,o)$. If $t_1(x),t_1(y),t_2(y)\in H$ for some $x,y\in G$ and $t_1,t_2\in T_n(G)$, then $t_1\in\rho_{_H}\langle x\rangle\cap\rho_{_H}\langle y\rangle$. Thus $\rho_{_H}\langle x\rangle\cap\rho_{_H}\langle y\rangle\neq\varnothing$, so
$\rho_{_H}\langle x\rangle=\rho_{_H}\langle y\rangle$. But $t_2(y)\in H$ means that $t_2\in\rho_{_H}\langle y\rangle$. Hence $t_2\in\rho_{_H}\langle x\rangle$, i.e., $t_2(x)\in H$. This proves $(b)$.

$(b)\longrightarrow (c)$. Suppose that $\rho^{\circ}_{_H}\langle
t_1\rangle\cap\rho^{\circ}_{_H}\langle t_2\rangle\neq\varnothing$ for some $t_1,t_2\in T_n(G)$.
 Then there exists $x\in G$ such that $x\in\rho^{\circ}_{_H}\langle t_1\rangle\cap\rho^{\circ}_{_H}
 \langle t_2\rangle$, i.e., $t_1(x)\in H$ and $t_2(x)\in H.$
 If $y\in\rho^{\circ}_{_H}\langle t_1\rangle$, i.e., $t_1(y)\in H,$ then, by \eqref{e8}, we have $t_2(y)\in H,$ i.e.,
$y\in\rho^{\circ}_{_H}\langle t_2\rangle$. So, $\rho^{\circ}_{_H}\langle t_1\rangle\subseteq\rho^{\circ}_{_H}\langle t_2\rangle$. Similarly
we obtain $\rho^{\circ}_{_H}\langle t_2\rangle\subseteq\rho^{\circ}_{_H}\langle t_1\rangle$. This proves $(c)$.

$(c)\longrightarrow (a)$. Let $\rho_{_H}\langle
a\rangle\cap\rho_{_H}\langle b\rangle\neq\varnothing$ for some
$a,b\in G.$ Then $t_1\in\rho_{_H}\langle a\rangle\cap\rho_{_H}
\langle b\rangle$ for some $t_1\in T_n(G)$, i.e., $t_1(a)$ and
$t_1(b)$ are in $H.$ Consequently, $a,b\in\rho^{\circ}_{_H}\langle
t_1\rangle$. Let $t_2$ is any translation from $\rho_{_H}\langle
a\rangle$, then $t_2(a)\in H$, i.e., $a\in\rho^{\circ}_{_H}\langle
t_2\rangle$. Therefore $\rho^{\circ}_{_H}\langle
t_1\rangle\cap\rho^{\circ}_{_H}\langle t_2\rangle\neq\varnothing$
which, by \eqref{e9}, gives $\rho^{\circ}_{_H}\langle
t_1\rangle=\rho^{\circ}_{_H}\langle t_2 \rangle$. As
$b\in\rho^{\circ}_{_H}\langle t_1\rangle$, then, obviously,
$b\in\rho^{\circ}_{_H}\langle t_2\rangle$, i.e., $t_2(b)\in H$,
whence $t_2\in\rho_{_H}\langle b\rangle$. So, $\rho_{_H}\langle
a\rangle\subseteq\rho_{_H}\langle b\rangle$. The reverse inclusion
follows similarly. Thus, $\rho_{_H}\langle
a\rangle=\rho_{_H}\langle b\rangle$. So, $H$ is strong.
\end{proof}

\noindent{\bf Remark.} The empty subset of a Menger algebra is strong since for $H=\varnothing$ the premise and the conclusion of the implication (\ref{e8}) are false, so it is true.

\begin{proposition}\label{P2.9}
Every nonempty strong subset $H$ of a Menger algebra $(G,o)$ is an $\mathcal{R}_H$-class different from $W_H$.
\end{proposition}
\begin{proof}
Let $H\neq\varnothing$ be a strong subset of a Menger algebra
$(G,o)$. Then it satisfies the condition $(b)$ of Theorem
\ref{T2.8}. Replacing $t_1$ in (\ref{e8}) by the identity
translation we see that $H$ is a normal $v$-complex. Proposition
\ref{P2.3} completes the proof.
\end{proof}

\begin{proposition}\label{P2.10}
Let $H$ be a strong subset of a Menger algebra $(G,o)$ of rank $n$. Then all $\mathcal{R}_H $-classes are  nonempty members of the family
$$
\mathcal{K}=\{\rho^{\circ}_{_H}\langle t\rangle\,|\,t\in T_n(G)\}\cup\{W_H\}.
$$
\end{proposition}
\begin{proof}
Firts, we show that $\mathcal{K}$ is a partition of $G$. Obviously, $\bigcup\mathcal{K}\subseteq G$ since $W_H\subset G$ and $\rho^{\circ}_{_H}\langle t\rangle\subset G$ for each $t\in T_n (G)$. Let $g$ be an arbitrary element of $G.$
Then either $t(g)\not\in H$ for all $t\in T_n(G)$, or $t(g)\in H$ for some $t\in T_n(G)$. In the first case
$g\in W_H$, in the second $g\in\rho^{\circ}_{_H}\langle t\rangle$ for some $t\in T_n(G)$. So,
$G\subseteq\bigcup\mathcal{K}$. Thus, $G=\bigcup\mathcal{K}$. This means that $\mathcal{K}$ covers $G.$

If $g\in\rho^{\circ}_{_H}\langle t_1\rangle\cap W_H$, where
$t_1\in T_n(G)$, then $t_1(g)\in H$ and $t(g)\not\in H$ for all
$t\in T_n(G)$. So, $t_1(g)\in H$ and $t_1(g)\not\in H.$ This is
impossible. Therefore $\rho^{\circ}_{_H}\langle t_1\rangle\cap
W_H=\varnothing$ for each $t_1\in T_n(G)$. Now, if
$g\in\rho^{\circ}_{_H}\langle
t_1\rangle\cap\rho^{\circ}_{_H}\langle t_2\rangle$, where
$t_1,t_2\in T_n(G)$, then $\rho^{\circ}_{_H}\langle
t_1\rangle\cap\rho^{\circ}_{_H}\langle t_2\rangle\neq\varnothing$.
Thus $\rho^{\circ}_{_H}\langle t_1\rangle=\rho^{\circ}_{_H}\langle
t_2\rangle$ because $H$ is a strong subset of $(G,o)$. So, if
$\rho^{\circ}_{_H}\langle t_1\rangle\neq\rho^{\circ}_{_H}\langle
t_2\rangle$, then $\rho^{\circ}_{_H}\langle
t_1\rangle\cap\rho^{\circ}_{_H}\langle t_2\rangle=\varnothing$.
Hence $\mathcal{K}$ is a partition of $G$.

Now let $(a,b)\in\mathcal{R}_H$, i.e., $\rho_{_H}\langle a\rangle=\rho_{_H}\langle b\rangle$. If $\rho_{_H}\langle a\rangle=\varnothing$, then $a,b\in W_H$, and as it is noted above,
$W_H$ is one of the $\mathcal{R}_H$-classes. If $\rho_{_H}\langle a\rangle\neq\varnothing$, then there is a translation $t\in T_n(G)$ such that $t\in\rho_{_H}\langle a\rangle=\rho_{_H}\langle b\rangle$, so $t(a)\in H$ and $t(b)\in H.$ Consequently, $a,b\in\rho^{\circ}_{_H}\langle t\rangle$.

Conversely, if $a,b\in\rho^{\circ}_{_H}\langle t\rangle$ for some $t\in T_n (G)$, then $t(a)\in H$ and $t(b)\in H$, so $t\in \rho_{_H}\langle a\rangle\cap\rho_{_H}\langle b\rangle$.
Thus, $\rho_{_H}\langle a\rangle\cap\rho_{_H}\langle b\rangle\neq\varnothing$, hence, $\rho_{_H}\langle
a\rangle=\rho_{_H}\langle b\rangle$, i.e., $(a,b)\in\mathcal{R}_H$. Therefore $\rho^{\circ}_{_H}\langle
t\rangle$ is the $\mathcal{R}_H$-class. In this way we have shown that $\mathcal{K}$ is a family of $\mathcal{R}_H$-classes.
\end{proof}

\begin{theorem}\label{T2.11}
Let $H$ be a strong subset of a Menger algebra $(G,o)$. If $X\neq W_H$ is an
$\mathcal{R}_H $-class, then $X$ is a strong subset of $(G,o)$ such that
$W_H\subseteq W_X$ and $\mathcal{R}_H\subseteq\mathcal{R}_X$.
Moreover, the restriction  of $\mathcal{R}_H$ to $G\!\setminus\!W_X$ coincides with
$\mathcal{R}^*_X$.
\end{theorem}
\begin{proof}
Since $\mathcal{R}_H$-class $X$ is different from $W_H$, there
exists $t\in T_n(G)$ such that $X=\rho^{\circ}_{_H}\langle
t\rangle$ (Proposition \ref{P2.10}). Let $y,z\in G$ and
$\rho_{_X}\langle y\rangle\cap\rho_{_X}\langle
z\rangle\neq\varnothing$. Then for some translation $t_1\in
T_n(G)$ we have $t_1\in\rho_{_X}\langle y\rangle\cap\rho_{_X}
\langle z\rangle$. This means that $t_1(y)$ and $t_1(z)$ are in
$\rho^{\circ}_{_H}\langle t \rangle$. So, $t(t_1(y))$ and
$t(t_1(z))$ are in $H.$ Consequently, $(t\circ t_1)(y)$ and
$(t\circ t_1)(z)$ also are in $H,$ i.e., $t\circ
t_1\in\rho_{_H}\langle y\rangle$ and $t\circ t_1\in
\rho_{_H}\langle z\rangle$, whence $\rho_{_H}\langle
y\rangle\cap\rho_{_H}\langle z\rangle\neq\varnothing$. So
$\rho_{_H}\langle y\rangle=\rho_{_H}\langle z\rangle$, because $H$
is a strong subset.
 Now let $t_2$ be an arbitrary
translation of $\rho_{_X}\langle y\rangle$. Then $t_2(y)\in X$,
i.e., $t_2 (y)\in\rho^{\circ}_{_H}\langle t\rangle$. Therefore,
$t(t_2(y))\in H$, which gives $t\circ t_2\in\rho_{_H}\langle
y\rangle=\rho_{_H}\langle z\rangle$. So, $t\circ
t_2\in\rho_{_H}\langle z\rangle$, i.e., $t(t_2(z))\in H.$ This
implies $t_2(z)\in\rho^{\circ}_{_H}\langle t\rangle$, which means
that $t_2(z)\in X$. Thus, $t_2\in\rho_{_X}\langle z\rangle$.
Therefore, $\rho_{_X}\langle y\rangle\subseteq\rho_{_X}\langle
z\rangle$. Similarly, we show that $\rho_{_X}\langle
z\rangle\subseteq\rho_{_X}\langle y\rangle$. So, $\rho_{_X}\langle
y\rangle=\rho_{_X}\langle z\rangle$. Hence, $\rho_{_X}\langle
y\rangle\cap\rho_{_X}\langle z\rangle\neq\varnothing$ implies
$\rho_{_X}\langle y\rangle=\rho_{_X}\langle z\rangle$. This means
that $X$ is a strong subset.

Now assume that $(a,b)\in\mathcal{R}_H$, i.e., $\rho_{_H}\langle a\rangle=\rho_{_H}\langle b\rangle$. Let
$t_3\in\rho_{_X}\langle a\rangle$, where $t_3\in T_n(G)$. Then $t_3(a)\in X,$ i.e., $t_3(a)\in\rho^{\circ}_{_H}\langle t\rangle$. Therefore, $(t\circ t_3)(a)\in H,$ so $t\circ t_3\in \rho_{H}\langle a\rangle=\rho_{_H}\langle b\rangle$. Thus, $t\circ t_3\in\rho_{_H}\langle b\rangle$, i.e., $t(t_3(b))\in H,$ whence $t_3(b)\in\rho^{\circ}_{_H}\langle t\rangle$. This gives $t_3(b)\in X.$ Consequently, $t_3\in \rho_{_X}\langle b\rangle$. Therefore, $\rho_{_X}\langle a\rangle\subseteq\rho_{_X}\langle b\rangle$.
Similarly, we show that $\rho_{_X}\langle b\rangle\subseteq\rho_{_X}\langle a\rangle$. Thus,
$\rho_{_X}\langle a\rangle=\rho_{_X}\langle b\rangle$, i.e., $(a,b)\in\mathcal{R}_X$. So,
$\mathcal{R}_H \subseteq \mathcal{R}_X$.

From this inclusion it follows that $W_H$ is contained in some $\mathcal{R}_X$-class.
Let $a\in W_H$, i.e., $\rho_{_H}\langle a\rangle=\varnothing$.
Then also $\rho_{_X}\langle a\rangle=\varnothing$. Indeed,
$\rho_{_X}\langle a\rangle\neq\varnothing$ means that $t_4(a)\in X=\rho^{\circ}_{_H}\langle t
\rangle$ for some $t_4\in\rho_{_X}\langle a\rangle$. So, $t_4(a)\in\rho^{\circ}_{_H}\langle t\rangle$, which implies $t(t_4(a))\in H.$ Thus, $(t\circ t_4)(a)\in H,$ and consequently, $t\circ t_4\in\rho_{_H}\langle a\rangle$. Hence, $\rho_{_H}\langle a\rangle\neq\varnothing$, i.e., $a\not\in W_H$. This impossible because by the assumption $a\in W_H$. Obtained contradiction proves that $\rho_{_X}\langle a\rangle=\varnothing$. Consequently, $a\in W_X$ and $W_H\subseteq W_X$.

Since $\mathcal{R}_H\subseteq \mathcal{R}_X$, the set $W_X$ is a
union of $\mathcal{R}_H$-classes. To prove that $\mathcal{R}_H$
and $\mathcal{R}^*_X$ are coincide on the set $G\!\setminus\!W_X$
it suffices to show that $(a,b)\in\mathcal{R}_X$ implies
$(a,b)\in\mathcal{R}_H$ for all $a,b\in G\!\setminus\!W_X$. Let
$(a,b)\in\mathcal{R}_X$ for some $a,b\in G\!\setminus\!W_X$. Then
$\rho_{_X}\langle a\rangle=\rho_{_X}\langle
b\rangle\neq\varnothing$, so there exists a translation $t_5\in
T_n(G)$ such that $t_5(a)$ and $t_5 (b)$ are in
$X=\rho^{\circ}_{_H}\langle t\rangle$. Therefore, $t(t_5(a))$ and
$t(t_5(b))$ are in $H.$ So, $(t\circ t_5)(a)$ and $(t\circ
t_5)(b)$ also are in $H.$ Thus, $t\circ t_5\in\rho_{_H}\langle
a\rangle\cap\rho_{_H}\langle b\rangle$, which shows
$\rho_{_H}\langle a\rangle\cap\rho_{_H}\langle
b\rangle\neq\varnothing$. Since $H$ is strong, the last implies
$\rho_{_H}\langle a\rangle=\rho_{_H}\langle b\rangle$. Therefore,
$(a,b)\in\mathcal{R}_H$. Thus,
$\mathcal{R}^*_X=\mathcal{R}_H\cap(G\!\setminus\!W_X)\times(G\!\setminus\!W_X)$.
\end{proof}

\begin{collolary}\label{C2.12}
Let $H$ be a strong subset of a Menger algebra $(G,o)$. Then, for any translation $t\in T_n(G)$, the set
$\rho^{\circ}_{_H}\langle t\rangle$ is strong.
\end{collolary}
\begin{proof}
It was noted above that the empty subset is strong, so if for some $t\in T_n(G)$, we have
$\rho^{\circ}_{_H}\langle t\rangle=\varnothing$, then, obviously, $\rho^{\circ}_{_H}\langle t\rangle$ is strong. If $\rho^{\circ}_{_H}\langle t\rangle\neq\varnothing$, then $\rho^{\circ}_{_H}\langle t\rangle$ is
an $\mathcal{R}_H$-class different from $W_H$ (Proposition \ref{P2.10}). Therefore, by Theorem \ref{T2.11},
$\rho^{\circ}_{_H}\langle t\rangle$ is a strong subset of $(G,o)$.
\end{proof}

\begin{proposition}\label{P2.13}
If a strong subset $H$ of a Memger algebra $(G,o)$ is its $l$-ideal, then $\mathcal{R}_H=\mathcal{R}_X$ for any $X\ne W_H$.
\end{proposition}
\begin{proof}
Let $H$ be a strong subset of $(G,o)$. By Theorem \ref{T2.11}, we have $W_H\subseteq W_X$, $\mathcal{R}_H\subseteq\mathcal{R}_X$ and $\mathcal{R}^*_X=\mathcal{R}_H\cap(G\!\setminus\!W_X)\times (G\!\setminus\!W_X)$. Thus, to prove $\mathcal{R}_H=\mathcal{R}_X$ it remains to show that $W_H=W_X$.
Due to the above it is sufficient to show only $W_X \subseteq W_H$.

Let $a\in W_X$. Then $\rho_{_X}\langle a\rangle=\varnothing$, i.e., $t(a)\not\in X$ for all $t\in T_n(G)$. Since $X\neq W_H$, for some $t_1\in T_n(G)$ we have $X=\rho^{\circ}_{_H}\langle t_1\rangle$ (Proposition \ref{P2.10}). Therefore, $t(a)\not\in\rho^{\circ}_{_H}\langle t_1\rangle$ for all $t\in T_n(G)$. Thus,
\begin{equation}\label{e10}
(t_1\circ t)(a)\not\in H \ \mbox{for all } t\in T_n(G).
\end{equation}
Suppose that $a\not\in W_H$, i.e., $\rho_{_H}\langle a\rangle\neq\varnothing$. Then $t_2(a)\in H$ for some
$t_2\in T_n(G)$, whence $t_1(t_2(a))\in t_1(H)$. So, $(t_1\circ t_2)(a)\in t_1(H)\subseteq H$ since $H$ is an $l$-ideal. Hence $(t_1\circ t_2)(a)\in H$, which contradics to (\ref{e10}). Therefore $a\in W_H$ and  $W_X\subseteq W_H$. This completes the proof.
\end{proof}

We say that a $v$-congruence $\varepsilon$ on a Menger algebra $(G,o)$ is \textit{partially $v$-cancellative with $v$-residue} $W,$ if $W$ is an $\varepsilon$-class of $(G,o)$ and
$$
 (u[\overline{w}\,|_ig_1], u[\overline{w}\,|_ig_2])\in\varepsilon\wedge
  u[\overline{w}\,|_ig_1]\not\in W\longrightarrow (g_1,g_2)\in\varepsilon
$$
for all $u,g_1,g_2\in G$, $\overline{w}\in G^n$, $i=1,\ldots,n$, i.e., if $\varepsilon\cap(G\!\setminus\!W)\times (G\!\setminus\!W)$ is $v$-can\-cellative.
In the case when $W$ is the empty set we obtain a $v$-cancellative relation.

\begin{proposition}\label{P2.14}
A principal $v$-congruence $\mathcal{R}_H$ on a Menger algebra $(G,o)$
is partially $v$-cancellative with $v$-residue $W_H$ if and only if for any translation $t\in T_n(G)$ the following implication
\begin{equation}\label{e11}
 (t(g_1), t(g_2))\in\mathcal{R}_H\wedge t(g_1)\not\in W_H\longrightarrow (g_1,g_2)\in\mathcal{R}_H
\end{equation}
is true.
\end{proposition}
\begin{proof}
If $\mathcal{R}_H$ and $W_H$ satisfy (\ref{e11}), then for $t(x)=u[\overline{w}\,|_ix]$, where
$u\in G,$ $\overline{w}\in G^n,$ $i\in\{1,\ldots,n\}$, we obtain
\begin{equation}\label{e12}
(u[\overline{w}\,|_ig_1], u[\overline{w}\,|_ig_2])\in\mathcal{R}_H\wedge
  u[\overline{w}\,|_ig_1]\not\in W_H\longrightarrow (g_1,g_2)\in\mathcal{R}_H ,
\end{equation}
which shows that  $\mathcal{R}_H$ is partially $v$-cancellative with $v$-residue $W_H$.

Conversely, let (\ref{e12}) and the premise of (\ref{e11}) be satisfied for $t\in T_n(G)$ of the form
$t(x)=u[\overline{w}\,|_it_1 (x)]$, where $t_1\in T_n(G)$, $u\in G$, $\overline{w}\in G^n$, $i\in\{1,\ldots,n\}$. Then, $(u[\overline{w}\,|_it_1(g_1)], u[\overline{w}\,|_it_1(g_2)])\in\mathcal{R}_H$ and $u[\overline{w}\,|_it_1(g_1)]\not\in W_H$, hence by
(\ref{e12}), we get $(t_1(g_1),t_1(g_2))\in \mathcal{R}_H$. Since $W_H$ is an $l$-ideal, from
$u[\overline{w}\,|_it_1(g_1)]\not\in W_H$ we obtain $t_1(g_1)\not\in W_H$. So, $(t_1(g_1),t_1(g_2))\in\mathcal{R}_H$ and $t_1(g_1)\not\in W_H$. Continuing a similar argumentation
after a finite number of steps we obtain $(g_1,g_2)\in\mathcal{R}_H$.
Thus, the condition (\ref{e11}) is proved.
\end{proof}

\begin{theorem}\label{T2.15}
A nonempty subset $H$ of a Menger algebra $(G,o)$ is strong if and only if it has the following two properties:
\begin{itemize}
  \item [$(i)$] \ $\rho_{_H}\langle h_1\rangle\cap\rho_{_H}\langle
  h_2\rangle\neq\varnothing$ for all $h_1,h_2\in H,$
  \item [$(ii)$] \ $\mathcal{R}_H$ is partially $v$-cancellative with $v$-residue $W_H$.
\end{itemize}
\end{theorem}
\begin{proof}
If a nonempty subset $H$ of $(G,o)$ is strong, then, by
Proposition \ref{P2.9}, it is an $\mathcal{R}_H$-class different from $W_H$. Since $H\neq W_H$,
by Proposition \ref{P2.10}, there exists a translation $t\in T_n(G)$ such that $H=\rho^{\circ}_{_H}\langle  t\rangle$. So, $t(h_1)\in H$ and $t(h_2)\in H$ for any $h_1,h_2\in H.$ Thus,
$t\in\rho_{_H}\langle h_1\rangle\cap\rho_{_H}\langle h_2\rangle$, which implies
$\rho_{_H}\langle h_1\rangle\cap\rho_{_H}\langle h_2\rangle\neq\varnothing$. This proves the first property.

Now suppose that $(u[\overline{w}\,|_ig_1],u[\overline{w}\,|_ig_2])\in\mathcal{R}_H$ and
$u[\overline{w}\,|_ig_1]\not\in W_H$ for some $u,g_1,g_2\in G,$ $\overline{w}\in G^n,$ $i=1,\ldots,n$. Since $u[\overline{w}\,|_ig_1]\not\in W_H$ means that $t_1(u[\overline{w}\,|_ig_1])\in H$ for some
$t_1\in T_n(G)$, we have $t_1(u[\overline{w}\,|_ig_2])\in H$. If $t(g_1)\in H$ for some $t\in T_n(G)$, then
from $t_1(u[\overline{w}\,|_ig_2]),t_1(u[\overline{w}\,|_ig_1]),t(g_1)\in H$ we obtain $t(g_2)\in H$ because $H$ is a strong subset. So, $t(g_1)\in H$ implies $t(g_2)\in H$. Analogously $t(g_2)\in H$ implies $t(g_1)\in H.$ Thus $(g_1,g_2)\in\mathcal{R}_H,$ which proves the second property.

Conversely, let $H$ be an arbitrary nonempty subset of $(G,o)$ satisfying $(i)$ and $(ii)$.
If $\rho_{_H}\langle a\rangle\cap\rho_{_H}\langle b\rangle\neq\varnothing$
for some $a,b\in G,$
 then for some $t\in T_n(G)$ we have $t\in\rho_{_H}\langle a\rangle$ and
  $t\in\rho_{_H}\langle b\rangle$. Thus $t(a)\in H$ and $t(b)\in H.$ But
  $H\cap W_H=\varnothing$, so $t(a)\not\in W_H$ and $t(b)\not\in W_H$. This, by $(i)$, gives
  $\rho_{_H}\langle t(a)\rangle\cap\rho_{_H}\langle t(b)\rangle\neq\varnothing$. Hence,
  there is a translation $t_1\in\rho_{_H}\langle t(a)\rangle\cap\rho_{_H}\langle t(b)\rangle$.
   For this translation  $t_1(t(a))\in H$ and $t_1(t(b))\in H.$  Consequently,
   $t(a)\in\rho^{\circ}_{_H}\langle t_1\rangle$ and
$t(b)\in\rho^{\circ}_{_H}\langle t_1\rangle$, whence
$(t(a),t(b))\in\mathcal{R}_H$  because, by Proposition
\ref{P2.10}, $\rho^{\circ}_{_H}\langle t_1\rangle$ is an
$\mathcal{R}_H$-class different from $W_H$. This, by $(ii)$,
implies $(a,b)\in\mathcal{R}_H$. Therefore, $\rho_{_H}\langle
a\rangle =\rho_{_H}\langle b\rangle$, which shows that a subset
$H$ is strong.
\end{proof}

We say that a subset $X$, possibly empty, of a Menger algebra $(G,o)$ is \textit{$l$-consistent} in
$(G,o)$, if for any $g\in G$ and $t\in T_n(G)$ \ $t(g)\in X$ implies $g\in X.$

\begin{proposition}\label{P2.16}
Let $H$ be a strong subset of a Menger algebra $(G,o)$. Then $\mathcal{R}_H$ is a $v$-cancellative $v$-congruence if and only if $W_H$ is $l$-consistent in $(G,o)$.
\end{proposition}
\begin{proof}
Let $H$ be a strong subset of $(G,o)$. Then, according to the second condition of Theorem \ref{T2.15}, the relation $\mathcal{R}_H$ is partially $v$-cancellative with $v$-residue $W_H$, i.e., the implication (\ref{e11}) is satisfied. Suppose that $t(a),t(b)\in W_H$. If $W_H$ is $l$-consistent in $(G,o)$, then $a,b\in W_H$,
so $(a,b)\in\mathcal{R}_H$. Thus, $\mathcal{R}_H$ satisfies $(\ref{e4})$, so it is $v$-cancellative.

Conversely, let $\mathcal{R}_H$ be a $v$-cancellative $v$-congruence. If $t(a)\in W_H$ for some $t\in T_n(G)$, then $t_1(t(a))\not\in H$ for all $t_1\in T_n(G)$. So, for all $t_1\in T_n(G)$ we have $(t_1\circ t)(t(a))\not\in H.$ Thus $t_1(t(t(a)))\not\in H,$ which implies $t(t(a))\in W_H$. Hence, $(t(a),t(t(a)))\in\mathcal{R}_H$, consequently $(a,t(a))\in\mathcal{R}_H$. Therefore $a\in W_H$. This shows that $W_H$ is $l$-consistent in $(G,o)$.
\end{proof}

\section{Principal $l$-congruences on Menger algebras}

For a nonempty subset $H$ of a Menger algebra $(G,o)$ of rank $n$ and $B=G^n\cup\{\overline{e}\}$ we define the following subsets:
\[
\eta_{_H}=\{(g,\overline{x})\in G\times B\,|\,g [\overline{x}]\in H\},
\]
\[
\eta_{_H}\langle g\rangle=\{\overline{x}\in B\,|\,g [\overline{x}]\in H \},
\]
\[
\eta^{\circ}_{_H}\langle\overline{x}\rangle=\{g\in G\,|\,g [\overline{x}]\in H\},
\]
\[
_H\!W=\{g\in G\,|\,\eta_{_H}\langle g\rangle=\varnothing\},
\]
\[
\mathcal{L}_H=\{(g_1, g_2)\in G\times G\,|\,\eta_{_H}\langle g_1\rangle=\eta_{_H}\langle g_2\rangle\}.
\]

It is easy to see that $\mathcal{L}_H$ is an equivalence relation on $(G,o)$ and $_H \!W$ is its an
equivalence class, provided that it is not empty. It is also easy show that
\begin{eqnarray}
&(g_1,g_2)\in\mathcal{L}_H\longleftrightarrow(\forall\overline{x}\in B)\;
\big(g_1[\overline{x}]\in H\longleftrightarrow g_2[\overline{x}]\in H\big),& \nonumber\\[4pt]
&g\in_H\!\!W\longleftrightarrow (\forall\overline{x}\in B)\;g[\overline{x}]\not\in H.& \nonumber
\end{eqnarray}
$_H\!W$ is called an \textit{$l$-residue} of $H$. Obviously $_H\!W\cap H=\varnothing$.

\begin{proposition}\label{P3.1}
For any subset $H$ of a Menger algebra $(G,o)$ the relation $\mathcal{L}_H$ is an $l$-congruence. Moreover, $_H\!W$ is an $s$-ideal of $(G,o)$, provided that $_H\!W\neq\varnothing$.
\end{proposition}
\begin{proof}
Let $(g_1,g_2)\in\mathcal{L}_H$, then for all $\overline{x},\overline{y}\in B$ we have
\[
g_1[\overline{y}*\overline{x}]\in H\longleftrightarrow g_2[\overline{y}*\overline{x}]\in H ,
\]
whence we obtain
\[
(g_1[\overline{y}])[\overline{x}]\in H
\longleftrightarrow (g_2[\overline{y}])[\overline{x}]\in H
\]
for all $\overline{x},\overline{y}\in B$. This means that $(g_1[\overline{y}],g_2[\overline{y}])\in\mathcal{L}_H$ for any $\overline{y}\in B$. Thus, $\mathcal{L}_H$ is
an $l$-regular, consequently, it is an $l$-congruence.

Let $g\in_H \!\!\!W$, then $g[\overline{y}*\overline{x}]\not\in H$ for all $\overline{y}*\overline{x}$, where $\overline{x},\overline{y}\in B$. Therefore, $(\forall\overline{x}\in B) (g[\overline{y}])[\overline{x}]\not\in
H$ for every $\overline{y}\in B.$ Thus, $g[\overline{y}]\in_H\!\!\!W$ for every $\overline{y}\in B$. So, $_H\!W$ is an $s$-ideal of $(G,o)$.
\end{proof}

Further the relation $\mathcal{L}_H$ will be called a \textit{principal $l$-congruence} induced by $H.$

\begin{proposition}\label{P3.2}
Any $l$-congruence $\varepsilon$ on a Menger algebra $(G,o)$ has the form
\[
\varepsilon = \bigcap \{\mathcal{L}_H\,|\,H\in G /\varepsilon\},
\]
where $G/\varepsilon$ is the set of all $\varepsilon$-classes.
\end{proposition}
\begin{proof}
Let $(g_1,g_2)\in\varepsilon$ and $g_1[\overline{x}]\in H$ for
some $\overline{x}\in B$ and $H\in G/\varepsilon$. Since $\varepsilon$ is an $l$-regular relation, for any $\overline{x}\in B$ we have $(g_1[\overline{x}], g_2[\overline{x}])\in\varepsilon$. But $H$ is an $\varepsilon$-class and $g_1[\overline{x}]\in H$, then, obviously, $g_2[\overline{x}]\in H$. Analogously,
$g_2[\overline{x}]\in H$ implies $g_1[\overline{x}]\in H$. Thus,
$$
(\forall\overline{x}\in B) \big(g_1[\overline{x}]\in H\longleftrightarrow g_2[\overline{x}]\in H\big),
$$
i.e., $(g_1,g_2)\in\mathcal{L}_H$. Since $H$ is an arbitrary $\varepsilon$-class, from the abowe we obtain inclusion $\varepsilon\subseteq\bigcap\{\mathcal{L}_H\,|\,H\in G/\varepsilon\}$.

To prove the converse inclusion consider an arbitrary pair $(g_1,g_2)\in\bigcap\{\mathcal{L}_H\,|\,H\in G/\varepsilon\}$. Then, $g_1[\overline{x}]\in H$ if and only if $g_2[\overline{x}]\in H$ for any $\overline{x}\in B$ and $H\in G/\varepsilon$. This, for $\overline{x}=\overline{e}$ and $H=\langle g_1\rangle$, where $\langle g_1\rangle$ denotes the $\varepsilon$-class determined by $g_1$, gives $g_2\in\langle g_1\rangle$ Thus, $(g_1,g_2)\in\varepsilon$ and consequently
$\bigcap\{\mathcal{L}_H\,|\,H\in G/\varepsilon\}\subseteq\varepsilon$, which completes the proof.
\end{proof}

\begin{definition}\rm
A subset $H$ of a Menger algebra $(G,o)$ is called \textit{$l$-strong} if
\begin{equation}\label{e13}
\eta_{_H}\langle a\rangle\cap\eta_{_H}\langle b\rangle\neq\varnothing\longrightarrow\eta_{_H}\langle
a\rangle=\eta_{_H}\langle b\rangle
\end{equation}
for all $a,b\in G.$
\end{definition}

\begin{theorem}\label{T3.4}
For a subset $H$ of a Menger algebra $(G,o)$ the following conditions are equivalent.
\begin{itemize}
  \item [$(i)$] \ $H$ is $l$-strong.
  \item [$(ii)$] \ For all $g_1,g_2\in G $ and $\overline{x},\overline{y}\in B$
\begin{equation}\label{e14}
  g_1[\overline{x}]\in H\wedge g_2[\overline{x}]\in H\wedge g_2[\overline{y}]\in H
  \longrightarrow g_1[\overline{y}]\in H.
\end{equation}
  \item [$(iii)$] \ For all $\overline{x},\overline{y}\in B$
 \begin{equation}\label{e15}
\eta^{\circ}_{_H}\langle\overline{x}\rangle\cap\eta^{\circ}_{_H}\langle\overline{y}\rangle\neq\varnothing
\longrightarrow\eta^{\circ}_{_H}\langle\overline{x}\rangle=\eta^{\circ}_{_H}\langle\overline{y}\rangle .
\end{equation}
\end{itemize}
\end{theorem}
\begin{proof}
$(i)\longrightarrow (ii)$ \ Let $H$ be a $l$-strong subset of
$(G,o)$ and $g_1[\overline{x}], g_2[\overline{x}],
g_2[\overline{y}]\in H$ for some $g_1,g_2\in G,$
$\overline{x},\overline{y}\in B$. Then $g_1[\overline{x}],
g_2[\overline{x}]\in H$ imply $\overline{x}\in\eta_{_H}\langle
g_1\rangle\cap\eta_{_H}\langle g_2\rangle$, so $\eta_{_H}\langle
g_1\rangle\cap\eta_{_H}\langle g_2\rangle\neq\varnothing$. This,
according to (\ref{e13}), gives $\eta_{_H}\langle
g_1\rangle=\eta_{_H}\langle g_2\rangle$. From
$g_2[\overline{y}]\in H$ we obtain
$\overline{y}\in\eta_{_H}\langle g_2\rangle$, which means that
$\overline{y}\in\eta_{_H}\langle g_1\rangle$. Thus,
$g_1[\overline{y}]\in H$. This proves $(ii)$.

$(ii)\longrightarrow (iii)$ \ If $\eta^{\circ}_{_H}\langle\overline{x}\rangle\cap \eta^{\circ}_{_H}\langle\overline{y}\rangle\neq\varnothing$, then for any $g\in\eta^{\circ}_{_H}\langle\overline{x}\rangle\cap\eta^{\circ}_{_H}\langle\overline{y}\rangle$ we have $g[\overline{x}]\in H$ and $g[\overline{y}]\in H$. If $g_1\in\eta^{\circ}_{_H}\langle\overline{x}\rangle$, then
$g_1[\overline{x}]\in H$. Thus, $g_1[\overline{x}], g[\overline{x}], g[\overline{y}]\in H,$ whence, by (\ref{e14}), we obtain $g_1[\overline{y}]\in H.$ Thus $g_1\in\eta^{\circ}_{_H}\langle\overline{y}\rangle$. So,
$\eta^{\circ}_{_H}\langle\overline{x}\rangle\subseteq\eta^{\circ}_{_H}\langle\overline{y}\rangle$.
Analogously, $\eta^{\circ}_{_H}\langle\overline{y}\rangle\subseteq\eta^{\circ}_{_H}\langle\overline{x}\rangle$.
Hence, $\eta^{\circ}_{_H}\langle\overline{x}\rangle=\eta^{\circ}_{_H}\langle\overline{y}\rangle$.
This proves $(iii)$.

$(iii)\longrightarrow (i)$ \ If $\eta_{_H}\langle
a\rangle\cap\eta_{_H}\langle b\rangle\neq\varnothing$, then for
some $\overline{x}\in B$ we have $a[\overline{x}]\in H$ and
$b[\overline{x}]\in H.$ If $\overline{y}\in\eta_{_H}\langle
a\rangle$, then $a[\overline{y}]\in H$, so from
$a[\overline{x}]\in H$ and $a[\overline{y}]\in H$ we obtain
$a\in\eta^{\circ}_{_H}\langle\overline{x}\rangle\cap
\eta^{\circ}_{_H}\langle\overline{y}\rangle$. Consequently,
$\eta^{\circ}_{_H}\langle\overline{x}\rangle\cap
\eta^{\circ}_{_H}\langle\overline{y}\rangle\neq\varnothing$,
whence, by (\ref{e15}), we deduce
$\eta^{\circ}_{_H}\langle\overline{x}\rangle=\eta^{\circ}_{_H}\langle\overline{y}\rangle$.
Since $b[\overline{x}]\in H$, we have
$b\in\eta^{\circ}_{_H}\langle\overline{x}\rangle$, so
$b\in\eta^{\circ}_{_H}\langle\overline{y}\rangle$, i.e.,
$b[\overline{y}]\in H.$ Thus $\overline{y}\in \eta_{_H}\langle
b\rangle$. Therefore $\eta_{_H}\langle
a\rangle\subseteq\eta_{_H}\langle b\rangle$. Similarly we can show
that $\eta_{_H}\langle b\rangle\subseteq\eta_{_H}\langle
a\rangle$. Hence $\eta_{_H}\langle a\rangle=\eta_{_H}\langle
b\rangle$, which means that $H$ is $l$-strong.
\end{proof}

Note that according to (\ref{e14}) the empty subset is $l$-strong.

\begin{proposition}\label{P3.5}
Each nonempty normal $l$-complex $H$ of a Menger algebra $(G,o)$ is an $\mathcal{L}_H$-equivalence class different from $l$-residue $_H\!W.$
\end{proposition}
\begin{proof}
Let $h_1,h_2\in H$ and $h_1[\overline{x}]\in H$ for some $\overline{x}\in B$. Then, by (\ref{e2}), we obtain
$h_2[\overline{x}]\in H$. Similarly, $h_2[\overline{x}]\in H$ implies $h_1[\overline{x}]\in H$. So, for all $\overline{x}\in B$ we have $h_1[\overline{x}]\in H\longleftrightarrow h_2[\overline{x}]\in H.$ Thus $(h_1,h_2)\in\mathcal{L}_H$. Consequently, $H$ is contained in some $\mathcal{L}_H$-class. Denote this
class by $X$. Hence, $H\subseteq X$.

Let $g$ be an arbitrary element of $X$ and $h\in H.$ Then obviously $(h,g)\in\mathcal{L}_H$, i.e.,
$h[\overline{x}]\in H\longleftrightarrow g[\overline{x}]\in H$ for all $\overline{x}\in B$. This for $\overline{x}=\overline{e}$ means that $h\in H\longleftrightarrow g\in H$, so $g\in H.$ Therefore,
$X\subseteq H.$ Consequently, $H=X.$ Since $H\cap\,_H\!W=\varnothing$, then $X\neq\,_H\!W.$
\end{proof}

\begin{collolary}\label{C3.6}
Each nonempty $l$-strong subset $H$ of a Menger algebra $(G,o)$ is an $\mathcal{L}_H$-equivalence class different from $l$-residue $_H\!W.$
\end{collolary}
\begin{proof}
Indeed, putting in (\ref{e14}) $\overline{x}=\overline{e}$ we can see that any $l$-strong subset is a normal
$l$-complex.
\end{proof}

\begin{proposition}\label{P3.7}
Let $H$ be an $l$-strong subset of a Menger algebra $(G,o)$. Then all $\mathcal{L}_H$-classes are nonempty members of the family
\[
\mathcal{E}=\{\eta^{\circ}_{_H}\langle\overline{x}\rangle\,|\,\overline{x}\in B\}\cup\{_H\!W \}.
\]
\end{proposition}
\begin{proof}
First, we show that $\mathcal{E} $ is a partition of $G$. Clearly, $\bigcup\mathcal{E}\subseteq G$ since $_H\!W\subset G$ and $\eta^{\circ}_{_H}\langle\overline{x}\rangle\subset G$ for each $\overline{x}\in B$. On the other hand, for an arbitrary element $g\in G$ we have either $(\forall\overline{x}\in B)\,g[\overline{x}]\not\in H,$ or $g[\overline{x}]\in H$ for some $\overline{x}\in B$, that is, either $g\in_H\!\!\!W,$ or $g\in\eta^{\circ}_{_H}\langle\overline{x}\rangle$ for some $\overline{x}\in B$. Therefore, $g\in\bigcup\mathcal{E}$, so $G\subseteq\bigcup\mathcal{E}$. Thus, $G=\bigcup\mathcal{E}$, i.e.,
$\mathcal{E}$ covers $G$.

Moreover, members of $\mathcal{E}$ are pairwise disjoint. Indeed,
if
$\eta^{\circ}_{_H}\langle\overline{y}\rangle\cap_H\!W\neq\varnothing$
for some $\overline{y}\in B$, then there exists an element $g\in
G$ such that $g\in\eta^{\circ}_{_H}\langle\overline{y}\rangle$ and
$g\in _H \!\!\!W$, i.e., $g[\overline{y}]\in H$ and
$(\forall\overline{x}\in B)\,g [\overline{x}]\not\in H,$ so for
$\overline{x}=\overline{y}$ we have $g[\overline{y}]\in H$ and
$g[\overline{y}]\not\in H$, which is impossible. Thus,
$\eta^{\circ}_{_H}\langle\overline{y}\rangle\cap_H
\!\!W=\varnothing$ for every $\overline{y}\in B$. In the case
$\eta^{\circ}_{_H}\langle\overline{x}\rangle\cap\eta^{\circ}_{_H}\langle\overline{y}\rangle\neq\varnothing$
we obtain
$\eta^{\circ}_{_H}\langle\overline{x}\rangle=\eta^{\circ}_{_H}\langle\overline{y}\rangle$
because $H$ is an $l$-strong subset. Thus, if
$\eta^{\circ}_{_H}\langle\overline{x}\rangle\ne\eta^{\circ}_{_H}\langle\overline{y}\rangle$,
then
$\eta^{\circ}_{_H}\langle\overline{x}\rangle\cap\eta^{\circ}_{_H}\langle\overline{y}\rangle=\varnothing$,
for all $\overline{x},\overline{y}\in B$. Hence all members of
$\mathcal{E}$ are pairwise disjoint subsets. Therefore
$\mathcal{E}$ is a partition of $G.$

Now let $(g_1,g_2)\in\mathcal{L}_H$, i.e., $\eta_{_H}\langle g_1\rangle=\eta_{_H}\langle g_2\rangle$. If $\eta_{_H}\langle g_1\rangle=\varnothing$, then $g_1,g_2\in {}_H\!W$, so $g_1$ and $g_2$ are in an $\mathcal{L}_H$-class $_H\!W.$ If $\eta_{_H}\langle g_1\rangle\neq\varnothing$, then there exists
$\overline{x}\in B$ such that $\overline{x}\in\eta_{_H}\langle g_1\rangle=\eta_{_H}\langle g_2\rangle$, so $g_1[\overline{x}]\in H$ and $g_2[\overline{x}]\in H$, hence $g_1,g_2\in\eta^{\circ}_{_H}\langle\overline{x}\rangle$.

Conversely, if $g_1,g_2\in\eta^{\circ}_{_H}\langle\overline{x}\rangle$ for some $\overline{x}\in B$, then $g_1[\overline{x}]\in H$ and $g_2[\overline{x}]\in H$, so $\overline{x}\in\eta_{_H}\langle
g_1\rangle\cap\eta_{_H}\langle g_2\rangle$. Consequently, $\eta_{_H}\langle g_1\rangle\cap\eta_{_H}\langle
g_2\rangle\neq\varnothing$, whence $\eta_{_H}\langle g_1\rangle=\eta_{_H}\langle g_2\rangle$, i.e.,
$(g_1,g_2)\in\mathcal{L}_H$. Thus, $\eta^{\circ}_{_H}\langle\overline{x}\rangle$ is an
$\mathcal{L}_H$-class.
\end{proof}

\begin{theorem}\label{T3.8}
Let $H$ be a nonempty $l$-strong subset of a Menger algebra
$(G,o)$. Then each $\mathcal{L}_H$-class $X\ne\, _H\!W$ is
$l$-strong, $_H\!W\subseteq_X\!\!W,$
$\mathcal{L}_H\subseteq\mathcal{L}_X$ and restrictions of
$\mathcal{L}_H$ and $\mathcal{L}_X$  on $G\!\setminus _X\!\!W$
coincide.
\end{theorem}
\begin{proof}
Let $\eta_{_X}\langle a\rangle\cap\eta_{_X}\langle b\rangle\neq\varnothing$, then there exists $\overline{x}\in B$ such that $\overline{x}\in\eta_{_X}\langle a\rangle\cap\eta_{_X}\langle b\rangle$. So, $a[\overline{x}]\in X$
and $b[\overline{x}]\in X$. Since $X\ne\, _H\!W$ is an $\mathcal{L}_H$-class, by
Proposition \ref{P3.7}, we have $X=\eta^{\circ}_{_H}\langle\overline{y}\rangle$ for some
$\overline{y}\in B$. Consequently, $a[\overline{x}]\in\eta^{\circ}_{_H}\langle\overline{y}\rangle$ and
$b[\overline{x}]\in \eta^{\circ}_{_H}\langle\overline{y}\rangle$, so
$a[\overline{x}][\overline{y}]\in H$ and $b[\overline{x}][\overline{y}]\in H,$ i.e.,
$a[\overline{x}*\overline{y}]\in H$ and $b[\overline{x}*\overline{y}]\in H.$ Thus $\overline{x}*\overline{y}\in \eta_{_H}\langle a\rangle$ and
$\overline{x}*\overline{y}\in\eta_{_H}\langle b\rangle$. Therefore
$\eta_{_H}\langle a\rangle\cap\eta_{_H}\langle b\rangle\neq\varnothing$. Since $H$ is $l$-strong, the last gives $\eta_{_H}\langle a\rangle=\eta_{_H}\langle b\rangle$. Now let $\overline{z}\in\eta_{_X}\langle a\rangle$, i.e., $a[\overline{z}]\in X=\eta^{\circ}_{_H}\langle\overline{y}\rangle$. Then $a[\overline{z}*\overline{y}]\in H,$ and consequently,
$\overline{z}*\overline{y}\in \eta_{_H}\langle a\rangle=\eta_{_H}\langle b\rangle$. Thus, $\overline{z}*\overline{y}\in \eta_{_H}\langle b\rangle$, so $b[\overline{z}*\overline{y}]\in H$, whence we obtain $b[\overline{z}]\in\eta^{\circ}_{_H}\langle\overline{y}\rangle=X$. This means that $\overline{z}\in\eta_{_X}\langle b\rangle$, so $\eta_{_X}\langle a\rangle\subseteq\eta_{_X}\langle b\rangle$.
Similarly, we show $\eta_{_X}\langle b\rangle\subseteq\eta_{_X}\langle a\rangle$. Hence,
$\eta_{_X}\langle a\rangle=\eta_{_X}\langle b\rangle$, proves that $X$ is $l$-strong.

Let $(a,b)\in\mathcal{L}_H$, i.e., $\eta_{_H}\langle a\rangle=\eta_{_H}\langle b\rangle$. If
$\overline{u}\in\eta_{_X}\langle a\rangle$, then $a[\overline{u}]\in X=\eta^{\circ}_{_H}\langle\overline{y}\rangle$. So, $a[\overline{u}*\overline{y}]\in H,$ consequently
$\overline{u}*\overline{y}\in\eta_{_H}\langle a\rangle$, whence
$\overline{u}*\overline{y}\in\eta_{_H}\langle b\rangle$, i.e.,
$b[\overline{u}*\overline{y}]\in H$. Thus, $b[\overline{u}]\in\eta^{\circ}_{_H}\langle\overline{y}\rangle=X$, so $\overline{u}\in \eta_{_X}\langle b\rangle$. Therefore, $\eta_{_X}\langle
a\rangle\subseteq \eta_{_X}\langle b\rangle$. Similarly we obtain the reverse inclusion. Hence $\eta_{_X}\langle a\rangle=\eta_{_X}\langle b\rangle$. This shows that $(a,b)\in\mathcal{L}_X$. Therefore $\mathcal{L}_H\subseteq\mathcal{L}_X$.

Thus, $_H\!W$ is contained in some $\mathcal{L}_X$-class. We show that it is the class $_X\!W$. Obviously
$\eta_{_H}\langle a\rangle=\varnothing$ for any $a\in_H\!\!W$. Then also $\eta_{_X}\langle a\rangle=\varnothing$. Indeed, if $\eta_{_X}\langle a\rangle\neq\varnothing$, then for some $\overline{w}\in B$ we have $\overline{w}\in\eta_{_X}\langle a\rangle$ and $a[\overline{w}]\in \eta^{\circ}_{_H}\langle\overline{y}\rangle$ because $X=\eta^{\circ}_{_H}\langle\overline{y}\rangle$. Thus
$a[\overline{w}*\overline{y}]\in H.$ Consequently,
$\overline{w}*\overline{y}\in \eta_{_H}\langle a\rangle$, which implies
$\eta_{_H}\langle a\rangle\neq\varnothing$, i.e., $a\not\in _H\!\!\!W$. Obtained contradiction proves that $\eta_{_X}\langle a\rangle=\varnothing$. So, $a\in_X\!\!W$ and $_H\!W\subseteq_X\!\!\!W$.

Since $\mathcal{L}_H\subseteq\mathcal{L}_X$, the set $_X \!W$ is a
union of some $\mathcal{L}_H$-classes. We show that restrictions
of $\mathcal{L}_H$ and $\mathcal{L}_X$ on $G\!\setminus _X\!\!W$
 coincide. For this we show that for all $a,b\in
G\setminus_X\!\!W $ such that $(a,b)\in\mathcal{L}_X$ we have
$(a,b)\in\mathcal{L}_H$. In fact, if $a,b\in G\setminus_X \!\!W$
and $(a,b)\in\mathcal{L}_X$, then $\eta_{_H}\langle
a\rangle=\eta_{_H}\langle b\rangle\neq\varnothing$. Thus
$a[\overline{v}]\in X$ and $b[\overline{v}]\in X$ for some
$\overline{v}\in B$, that is,
$a[\overline{v}]\in\eta^{\circ}_{_H}\langle\overline{y}\rangle$
and
$b[\overline{v}]\in\eta^{\circ}_{_H}\langle\overline{y}\rangle$.
Therefore, $a[\overline{v}*\overline{y}]\in H$ and
$b[\overline{v}*\overline{y}]\in H$, which implies
$\overline{v}*\overline{y}\in\eta_{_H}\langle
a\rangle\cap\eta_{_H}\langle b\rangle$. So, $\eta_{_H}\langle
a\rangle\cap\eta_{_H}\langle b\rangle\neq\varnothing$. Since $H$
is $l$-strong, the last implies $\eta_{_H}\langle
a\rangle=\eta_{_H}\langle b\rangle$, i.e.,
$(a,b)\in\mathcal{L}_H$. Thus,
\[
\mathcal{L}_H \cap (G\setminus _X \!\!W)\times(G\setminus_X\!\!W)=\mathcal{L}_X\cap(G\setminus_X\!\!W)\times (G\setminus_X\!\!W),
\]
as required.
\end{proof}

\begin{collolary}\label{C3.9}
If $H$ is an $l$-strong subset of a Menger algebra $(G,o)$, then for every $\overline{x}\in B$ the set
$\eta^{\circ}_{_H}\langle\overline{x}\rangle$ is $l$-strong.
\end{collolary}
\begin{proof}
Since the empty subset is $l$-strong, so if $\eta^{\circ}_{_H}\langle\overline{x}\rangle=\varnothing$ for some $\overline{x}\in B$, then, obviously, $\eta^{\circ}_{_H}\langle\overline{x}\rangle$ is $l$-strong. If $\eta^{\circ}_{_H}\langle\overline{x}\rangle\neq\varnothing$, then $\eta^{\circ}_{_H}\langle
\overline{x}\rangle$ is an $\mathcal{L}_H$-class different from $_H \!W$ (Proposition \ref{P3.7}). Therefore, by Theorem \ref{T3.8}, $\eta^{\circ}_{_H}\langle\overline{x}\rangle$ is $l$-strong.
\end{proof}

\begin{proposition}\label{P3.10}
Each $\varepsilon$-class $X$ of an $l$-cancellative $l$-congruence
on a Menger algebra $(G,o)$ is $l$-strong and
$\varepsilon\subseteq\mathcal{L}_X$. Moreover, $\varepsilon$ and
$\mathcal{L}_X$  coincide  on $G\!\setminus_X\!\!W$.
\end{proposition}
\begin{proof}
Let $\eta_{_X}\langle a\rangle\cap\eta_{_X}\langle b\rangle\neq\varnothing$, then $a[\overline{x}]\in X$ and $b[\overline{x}]\in X$ for some $\overline{x}\in B$. Thus, $(a[\overline{x}], b[\overline{x}])\in\varepsilon$, whence by $l$-cancellativity we obtain $(a,b)\in\varepsilon$. This, in view of Proposition \ref{P3.2}, gives $\varepsilon\subseteq\mathcal{L}_X$. Therefore, $(a,b)\in\mathcal{L}_X$, so $\eta_{_X}\langle
a\rangle=\eta_{_X}\langle b\rangle$. Thus, $X$ is an $l$-strong subset of $(G,o)$.

Since $\varepsilon\subseteq\mathcal{L}_X$, we have
$$
\varepsilon\cap(G\!\setminus_X\!\!W)\times (G\!\setminus_X\!\!W)\subseteq
\mathcal{L}_X\cap(G\!\setminus_X\!\!W)\times(G\!\setminus_X\!\!W).
$$
Conversely, if $(a,b)\in\mathcal{L}_X$ and $a,b\in G\!\setminus_X\!\!W$, then
$$
(\forall\overline{x}\in B)\big(a[\overline{x}]\in X\longleftrightarrow b[\overline{x}]\in X\big).
$$
Since $a\not\in_X\!\!\!W$, we have $a[\overline{y}]\in X$ for some
$\overline{y}\in B$, the above means that also $b[\overline{y}]\in
X$. Consequently,
$(a[\overline{y}],b[\overline{y}])\in\varepsilon$. From this, by
$l$-cancellativity, we deduce $(a,b)\in\varepsilon$. So,
$$
\mathcal{L}_X\cap(G\!\setminus_X\!\!W)\times(G\!\setminus_X\!\!W)\subseteq
\varepsilon\cap(G\!\setminus_X \!\!W)\times(G\!\setminus_X\!\!W).
$$
Comparing these two inclusions we obtain
$$
\varepsilon\cap(G\!\setminus_X\!\!W)\times(G\!\setminus_X\!\!W)=
\mathcal{L}_X\cap(G\!\setminus_X\!\!W)\times(G\!\setminus_X\!\!W).
$$
This means that $\varepsilon$ and $\mathcal{L}_X$  coincide on the
set $G\!\setminus_X\!\!W$.
\end{proof}

\section{Principal congruences on Menger algebras}

For a nonempty subset $H$ of a Menger algebra $(G,o)$ of rank $n$ and $B=G^n\cup\{\overline{e}\}$ consider the following subsets:
\[
\sigma_{\!_H}=\{(g,(\overline{x},t))\in G\times(B\times T_n(G))\,|\,t(g[\overline{x}])\in H\},
\]
\[
\sigma_{\!_H}\langle g\rangle=\{(\overline{x},t)\in B\times T_n(G)\,|\,t(g[\overline{x}])\in H\},
\]
\[
\sigma^{\circ}_{\!_H}\langle\overline{x},t\rangle=\{g\in G\,|\,t(g[\overline{x}])\in H\},
\]
\[
W^H=\{g\in G\,|\,\sigma_{\!_H}\langle g\rangle=\varnothing\},
\]
\[
\mathcal{P}_H=\{(g_1,g_2)\in G\times G\,|\,\sigma_{\!_H}\langle g_1\rangle =
\sigma_{\!_H}\langle g_2\rangle\}.
\]
It is clear that $\mathcal{P}_H$ is an equivalence relation on $(G,o)$ and $W^H$ is its an equi\-valence class, provided that it is not empty. Obviously, for any $g_1,g_2,g\in G$ we have:
\begin{eqnarray}
& (g_1,g_2)\in\mathcal{P}_H\longleftrightarrow(\forall t\in T_n(G))(\forall\overline{x}\in B)\; \Big(t(g_1[\overline{x}])\in H\longleftrightarrow t(g_2[\overline{x}])\in H\Big), &\nonumber\\[4pt]
& g\in W^H\longleftrightarrow(\forall t\in T_n(G))(\forall\overline{x}\in B)\;t(g[\overline{x}])\not\in H. & \nonumber
\end{eqnarray}
The set $W^H$ is called \textit{biresidue} of $H.$ Note that $H\cap W^H=\varnothing$ since in the case $H\cap W^H\neq\varnothing$ there is an element $g\in G$ such that $g\in H$ and $t(g[\overline{x}])\not\in H$ for all $t\in T_n(G)$, $\overline{x}\in B$. So, for the identity translation $t$ and $\overline{x}=\overline{e}$ we obtain $g\in H$ and $g\not\in H,$  which is impossible.

\begin{proposition}\label{P4.1}
For any subset $H$ of a Menger algebra $(G,o)$ the relation $\mathcal{P}_H$ is a congruence. Moreover, if the set $W^H$ is nonempty, then it is an $sl$-ideal of $(G,o)$.
\end{proposition}
\begin{proof}
Let $(g_1,g_2)\in\mathcal{P}_H$. Then for all $t\in T_n(G)$ and $\overline{x}\in B$ we have
\begin{equation}\label{e16}
t(g_1[\overline{x}])\in H\longleftrightarrow t(g_2[\overline{x}])\in H.
\end{equation}
In particular, for $y_1,\ldots,y_n\in G$ we obtain $(y_1[\overline{x}],\ldots,y_n[\overline{x}])\in B$ and
$$
t(g_1[y_1[\overline{x}]\ldots y_n[\overline{x}])\in H\longleftrightarrow
t(g_2[y_1[\overline{x}]\ldots y_n[\overline{x}])\in H,
$$
 whence applying (\ref{e1}) we get
$$
t(g_1[\overline{y}][\overline{x}])\in H\longleftrightarrow
t(g_2[\overline{y}][\overline{x}])\in H,
$$
where $\overline{y}=(y_1,\ldots,y_n)$. Thus, $(g_1[\overline{y}],
g_2[\overline{y}])\in\mathcal{P}_H$. So, the relation
$\mathcal{P}_H$ is $l$-regu\-lar.

Since (\ref{e16}) holds for all $t\in T_n(G)$ and $\overline{x}\in B$, it also holds for all polynomials
$t(x)=t_1(u[\overline{w}_0\,|_i\,x])$, where $t_1\in T_n(G)$, $u\in G,$ $\overline{w}_0=(w_1[\overline{x}],\ldots,w_n[\overline{x}])\in G^n,$
$w_1,\ldots,w_n\in G,$ $i\in\{1,\ldots,n\}$. Therefore,
$$
t_1(u[\overline{w}_0\,|_i\,g_1[\overline{x}]])\in H\longleftrightarrow
t_1(u[\overline{w}_0\,|_i\,g_2[\overline{x}]])\in H,
$$
which in view of superassociativity can be rewritten in the form:
$$
t_1(u[\overline{w}\,|_i\,g_1][\overline{x}])\in H\longleftrightarrow
t_1(u[\overline{w}\,|_i\,g_2][\overline{x}])\in H,
$$
where $\overline{w}=(w_1,\ldots,w_n)$. So, $(u[\overline{w}\,|_i\,g_1],u[\overline{w}\,|_i\,g_2])\in\mathcal{P}_H$ for
all $i=1,\ldots,n$. Hence, $\mathcal{P}_H$ is a $v$-regular relation.

Thus, $\mathcal{P}_H$ is $l$-regular and $v$-regular, so it is stable, that is, $\mathcal{P}_H$ is a
congruence. In a similar way we can verify that $W^H$ is an $sl$-ideal.
\end{proof}

The relation $\mathcal{P}_H$ will be called the \textit{principal congruence} induced by $H.$

\begin{proposition}\label{P4.2}
Any congruence $\varepsilon$ on a Menger algebra $(G,o)$ has the form
$$
\varepsilon=\bigcap\{\mathcal{P}_H\,|\,H\in G/\varepsilon\},
$$
where $G/\varepsilon$ is the set of all $\varepsilon$-classes.
\end{proposition}
\begin{proof}
Let $(g_1,g_2)\in\varepsilon$ and $t(g_1[\overline{x}])\in H,$ where $H\in G/\varepsilon$, $t\in T_n(G)$ and $\overline{x}\in B$. As $\varepsilon$ is a congruence, from $(g_1,g_2)\in\varepsilon$, by $l$-regularity, we obtain $(g_1[\overline{x}],g_2[\overline{x}])\in\varepsilon$ for each $\overline{x}\in B,$ which, by $v$-regularity, gives $(t(g_1[\overline{x}]),t(g_2 [\overline{x}]))\in\varepsilon$. Hence, $t(g_1[\overline{x}])$ and $t(g_2[\overline{x}])$ are in the same $\varepsilon$-class. So, $t(g_2[\overline{x}])\in H.$ Analogously, from $t(g_2[\overline{x}])\in H$ we conclude $t(g_1[\overline{x}])\in H.$ Thus, $t(g_1[\overline{x}])\in H\longleftrightarrow t(g_2[\overline{x}])\in H$ for all $t\in T_n(G)$ and $\overline{x}\in B.$ This means that $(g_1,g_2)\in\mathcal{P}_H$. Since $H$ is an arbitrary $\varepsilon$-class, the above proves the inclusion $\varepsilon\subseteq\bigcap\{\mathcal{P}_H\,|\,H\in G/\varepsilon\}$.

Conversely, if $(g_1,g_2)\in\bigcap\{\mathcal{P}_H\,|\,H\in G/\varepsilon\}$, then for all $H\in G/\varepsilon$, each $t\in T_n(G)$ and any $\overline{x}\in B$ we have
$$
t(g_1[\overline{x}])\in H\longleftrightarrow t(g_2[\overline{x}])\in H,
$$
whence for $\overline{x}=\overline{e}$ we get
$$
t(g_1)\in H\longleftrightarrow t(g_2)\in H.
$$
Since this holds for each $t\in T_n(G)$ and $H\in G/\varepsilon$, for the identity
 translation we obtain $g_1\in H\longleftrightarrow g_2\in H.$ So,
 $g_2\in\langle g_1\rangle$, where $\langle g_1\rangle$ is the
 $\varepsilon$-class of $g_1$. Thus, $(g_1,g_2)\in\varepsilon$.
 Consequently, $\bigcap\{\mathcal{P}_H\,|\,H\in G/\varepsilon\}\subseteq\varepsilon$.
 This completes the proof.
\end{proof}

\begin{definition}\label{D4.3}\rm
A subset $H$ of a Menger algebra $(G,o)$ is called \textit{bistrong} if
\begin{equation}\label{e17}
\sigma_{\!_H}\langle a\rangle\cap\sigma_{\!_H}\langle b\rangle\neq\varnothing\longrightarrow \sigma_{\!_H}\langle a\rangle=\sigma_{\!_H}\langle b\rangle
\end{equation}
for any $a,b\in G$.
\end{definition}

\begin{theorem}\label{T4.4}
For a subset $H$ of a Menger algebra $(G,o)$ the following conditions are equivalent.
\begin{itemize}
  \item [$(i)$] \ $H$ is bistrong.
  \item [$(ii)$] For all $g_1,g_2\in G,$ $\overline{x},\overline{y}\in B$ and $t_1,t_2\in T_n(G)$
 \begin{equation}\label{e18}
 t_1(g_1[\overline{x}])\in H\wedge t_1(g_2[\overline{x}])\in H\wedge t_2(g_2[\overline{y}])\in H
\longrightarrow t_2(g_1[\overline{y}])\in H .
\end{equation}
  \item [$(iii)$] \ For all $\overline{x},\overline{y}\in B$ and $t_1,t_2\in T_n (G)$
\begin{equation}\label{e19}
\sigma^{\circ}_{\!_H}\langle\overline{x},t_1\rangle\cap\sigma^{\circ}_{\!_H}\langle\overline{y},t_2\rangle  \neq\varnothing\longrightarrow\sigma^{\circ}_{\!_H}\langle\overline{x},t_1\rangle= \sigma^{\circ}_{\!_H}\langle\overline{y},t_2\rangle .
\end{equation}
\end{itemize}
\end{theorem}
\begin{proof}
$(i)\longrightarrow (ii)$ \ Let $H$ be a bistrong subset of $(G,o)$. Suppose that $t_1(g_1[\overline{x}]), t_1(g_2[\overline{x}]), t_2(g_2[\overline{y}])\in H$ for some $g_1,g_2\in G,$ $\overline{x},\overline{y}\in B$ and $t_1,t_2\in T_n(G)$. Then from $t_1(g_1[\overline{x}]),t_1(g_2[\overline{x}])\in H$ we obtain
$(\overline{x},t_1)\in\sigma_{\!_H}\langle g_1\rangle\cap\sigma_{\!_H}\langle g_2\rangle$, so
$\sigma_{\!_H}\langle g_1\rangle\cap\sigma_{\!_H}\langle g_2\rangle\neq\varnothing$. This, by (\ref{e17}),
gives $\sigma_{\!_H}\langle g_1\rangle=\sigma_{\!_H}\langle g_2\rangle$. But $t_2(g_2[\overline{y}])\in H,$ hence $(\overline{y},t_2)\in\sigma_{\!_H}\langle g_2\rangle$. Consequently,
$(\overline{y},t_2)\in\sigma_{\!_H}\langle g_1\rangle$, i.e.,
$t_2(g_1[\overline{y}])\in H.$ This proves (\ref{e18}) and $(ii)$.

$(ii)\longrightarrow (iii)$ \ Let $\sigma^{\circ}_{\!_H}\langle\overline{x},t_1\rangle\cap\sigma^{\circ}_{\!_H}\langle\overline{y},
t_2\rangle\neq\varnothing$ for some $\overline{x},\overline{y}\in B,$ $t_1,t_2\in T_n(G)$. Then
there is an element $g\in G$ such that $g\in\sigma^{\circ}_{\!_H}\langle\overline{x}, t_1\rangle$ and
$g\in\sigma^{\circ}_{\!_H}\langle\overline{y}, t_2\rangle$, so,
$t_1(g[\overline{x}])\in H$ and $t_2(g[\overline{y}])\in H.$ Let
$g_1\in\sigma^{\circ}_{\!_H}\langle\overline{x},t_1\rangle$, i.e., $t_1(g_1[\overline{x}])\in H.$ Thus
$t_1(g_1[\overline{x}]),t_1(g[\overline{x}]),t_2(g[\overline{y}])\in H,$ whence, by (\ref{e18}), we obtain $t_2(g_1[\overline{y}])\in H.$ Consequently, $g_1\in\sigma^{\circ}_{\!_H}\langle\overline{y},t_2\rangle$, which proves the inclusion $\sigma^{\circ}_{\!_H}\langle\overline{x},t_1\rangle\subseteq \sigma^{\circ}_{\!_H}\langle\overline{y},t_2\rangle$. The proof of the reverse inclusion is similar. Thus, $\sigma^{\circ}_{\!_H}\langle\overline{x},t_1\rangle=\sigma^{\circ}_{\!_H}\langle\overline{y},t_2\rangle$.
This proves \eqref{e19} and $(iii)$.

$(iii)\longrightarrow (i)$ \ If the premise of \eqref{e17} is satisfied, then there exists $(\overline{x},t)$ such that $(\overline{x},t)\in\sigma_{\!_H}\langle a\rangle\cap\sigma_{\!_H}\langle b\rangle$. Therefore, $t(a[\overline{x}])\in H$ and $t(b[\overline{x}])\in H.$ Consequently,
$a\in\sigma^{\circ}_{\!_H}\langle\overline{x},t\rangle$ and $b\in\sigma^{\circ}_{\!_H}\langle\overline{x},t\rangle$. Let
$(\overline{y},t_1)\in\sigma_{\!_H}\langle a\rangle$. Then $t_1(a[\overline{y}])\in H,$ whence $a\in\sigma^{\circ}_{\!_H}\langle\overline{y},t_1\rangle$. Therefore, $\sigma^{\circ}_{\!_H}\langle\overline{x},t\rangle\cap\sigma^{\circ}_{\!_H}\langle\overline{y},
t_1\rangle\neq\varnothing$, which, by $\eqref{e19}$, gives $\sigma^{\circ}_{\!_H}\langle\overline{x},t\rangle=\sigma^{\circ}_{\!_H}\langle\overline{y}, t_1\rangle$. So, $b\in\sigma^{\circ}_{\!_H}\langle\overline{y},t_1\rangle$, i.e., $t_1(b[\overline{y}])\in H.$ Thus, $(\overline{y},t_1)\in\sigma_{\!_H}\langle b\rangle$ and $\sigma_{\!_H}\langle a\rangle\subseteq \sigma_{\!_H}\langle b\rangle$. Analogously we can prove $\sigma_{\!_H}\langle b\rangle\subseteq\sigma_{\!_H}\langle a\rangle$. Therefore, $\sigma_{\!_H}\langle a\rangle=\sigma_{\!_H}\langle b\rangle$. This proves (\ref{e17}).
\end{proof}

Note that according to Theorem \ref{T4.4} the empty subset of a Menger algebra is bistrong since for
$H=\varnothing$ the premise and the conclusion of (\ref{e18}) are false, so this implication is true.

\begin{proposition}\label{P4.5}
Every nonempty normal bicomplex $H$ of a Menger algebra $(G,o)$ is a $\mathcal{P}_H$-class
different from $W^H.$
\end{proposition}
\begin{proof}
Let $h_1,h_2\in H$ and $t(h_1[\overline{x}])\in H$ for some $t\in T_n(G)$, $\overline{x}\in B.$ Then, by (\ref{e3}), we have $t(h_2[\overline{x}])\in H.$ Similarly, from $t(h_2[\overline{x}])\in
H$ we obtain $t(h_1[\overline{x}])\in H.$ Thus,
\[
(\forall t\in T_n (G)) (\forall\overline{x}\in B) \Big(t(h_1[\overline{x}])\in H\longleftrightarrow
t(h_2[\overline{x}])\in H\Big),
\]
so $(h_1,h_2)\in\mathcal{P}_H.$ Consequently, $H$ is contained in some $\mathcal{P}_H$-class. Denote this class by $X$. So, $H\subseteq X$. Now let $g\in X$ and $h\in H.$ Then, obviously, $(h,g)\in\mathcal{P}_H$, i.e.,
$$
 t(h[\overline{x}])\in H\longleftrightarrow t(g[\overline{x}])\in H
$$
for all $t\in T_n(G)$, $\overline{x}\in B.$ From this, for the identity translation $t$ and $\overline{x}=\overline{e}$, we get $h\in H\longleftrightarrow g\in H.$ So, $g\in H$ and $X\subseteq H,$ whence
$H=X$. Since $H\cap W^H=\varnothing$, the class $X=H$ is different from $W^H.$
\end{proof}

\begin{collolary}\label{C4.6}
Every nonempty bistrong subset $H$ of a Menger algebra $(G,o)$ is an equivalence class of the principal congruence $\mathcal{P}_H$ different from biresidue $W^H.$
\end{collolary}
\begin{proof}
Every bistrong subset $H\ne\varnothing$ satisfies (\ref{e18}) for any elements $g_1,g_2\in G$ and
$\overline{x},\overline{y}\in B,$ $t_1,t_2\in T_n(G)$. Replacing $t_1$ by the identity translation and $\overline{x}$ by $\overline{e}$ we see that $H$ is a normal bicomplex. Hence, by Proposition \ref{P4.5},
it is a $\mathcal{P}_H$-class different from $W^H.$
\end{proof}

\begin{proposition}\label{P4.7}
Let $H$ be a bistrong subset of a Menger algebra $(G,o)$. Then the $\mathcal{P}_H$-classes are nonempty members of the family
$$
\mathcal{D}=\{\sigma^{\circ}_{_H}\langle\overline{x},t\rangle\,|\,(\overline{x},t)\in B\times T_n(G)\}\cup\{W^H\}.
$$
\end{proposition}
\begin{proof}
First, make sure that the family $\mathcal{D}$ is a partition of the set $G.$ In fact, $W^H\subseteq G$ and
$\sigma^{\circ}_{_H}\langle\overline{x},t\rangle\subseteq G$ for all $t\in T_n(G)$ and $\overline{x}\in B.$ Thus, $\bigcup\mathcal{D}\subseteq G.$ Conversely, for an arbitrary element $g\in G$ we have two possibilities: either $(\forall t\in T_n(G))(\forall\overline{x}\in B)\,t(g[\overline{x}])\not\in H,$ or $t(g[\overline{x}])\in H$ for some $t\in T_n(G)$ and $\overline{x}\in B$. In the first case $g\in W^H$, in the second $g\in\sigma^{\circ}_{_H}\langle\overline{x},t\rangle$ for some $t\in T_n(G)$ and $\overline{x}\in B.$
Therefore, $g\in\bigcup\mathcal{D}$, so $G\subseteq\bigcup\mathcal{D}$. Thus,
$G=\bigcup\mathcal{D}$, i.e., $\mathcal{D}$ covers $G.$

Next, we show that the members of $\mathcal{D}$ are pairwise disjoint. Indeed, if there is $g\in G$ such that $g\in\sigma^{\circ}_{_H}\langle\overline{x}_1,t_1\rangle\cap W^H,$ where $t_1\in T_n(G)$, $\overline{x}_1\in
B,$ then $t_1(g[\overline{x}_1])\in H$ and $t(g[\overline{x}])\not\in H$ for all $t\in T_n(G)$, $\overline{x}\in B$, so $t_1(g[\overline{x}_1])\in H$ and $t_1(g[\overline{x}_1])\not\in H$ at the same time, which is impossible. Therefore, $\sigma^{\circ}_{_H}\langle\overline{x}_1,t_1\rangle\cap W^H =\varnothing$ for all $t_1\in T_n(G)$, $\overline{x}_1\in B.$ If  $\sigma^{\circ}_{_H}\langle\overline{x}_1,t_1\rangle\cap\sigma^{\circ}_{_H}\langle\overline{x}_2,t_2\rangle \neq\varnothing$, then, according to (\ref{e19}), $\sigma^{\circ}_{_H}\langle\overline{x}_1, t_1\rangle=\sigma^{\circ}_{_H}\langle\overline{x}_2,t_2\rangle$. So, if
$\sigma^{\circ}_{_H}\langle\overline{x}_1,t_1\rangle\neq\sigma^{\circ}_{_H}\langle\overline{x}_2,t_2\rangle$, then $\sigma^{\circ}_{_H}\langle\overline{x}_1,t_1\rangle\cap\sigma^{\circ}_{_H}\langle\overline{x}_2, t_2\rangle=\varnothing$. Hence, $\mathcal{D}$ is a partition of $G.$

Now let $(a,b)\in\mathcal{P}_H$, i.e., $\sigma_{\!_H}\langle a\rangle=\sigma_{\!_H}\langle b\rangle$. If
$\sigma_{\!_H}\langle a\rangle=\varnothing$, then $a,b\in W^H.$ If
$\sigma_{\!_H}\langle a\rangle\neq\varnothing$, then there exist $t\in T_n(G)$ and $\overline{x}\in B$ such that $(\overline{x},t)\in\sigma_{\!_H}\langle a\rangle=\sigma_{\!_H}\langle b\rangle$, so $t(a[\overline{x}])\in H$ and $t(b[\overline{x}])\in H.$ Hence, $a,b\in\sigma^{\circ}_{_H}\langle
\overline{x},t\rangle$. This shows that any $a,b\in G$ such that $(a,b)\in\mathcal{P}_H$ belong to the same equivalence class of the relation $\mathcal{P}_H$.

Conversely, if $a,b\in\sigma^{\circ}_{_H}\langle\overline{x},t\rangle$ for some $t\in T_n(G)$ and $\overline{x}\in B,$ then $t(a[\overline{x}])\in H$ and $t(b[\overline{x}])\in H,$ so
$(\overline{x},t)\in\sigma_{\!_H}\langle a\rangle\cap\sigma_{\!_H}\langle b\rangle$. Consequently,
$\sigma_{\!_H}\langle a\rangle\cap\sigma_{\!_H}\langle b\rangle\neq\varnothing$, so $\sigma_{\!_H}\langle
a\rangle=\sigma_{\!_H}\langle b\rangle$ because $H$ is bistrong. Thus, $(a,b)\in \mathcal{P}_H$. Hence $\sigma^{\circ}_{_H}\langle\overline{x},t\rangle$ is an equivalence class of the relation $\mathcal{P}_H$.
\end{proof}

To any polynomial $t\in T_n(G)$ and any vector $\overline{a}\in B$ we associate the polynomial
$t^{\overline{a}}\in T_n (G)$ obtained from $t$ by assigning $[\overline{a}]$ to all elements standing in the square brackets except for the variable $x$ and the elements standing directly
before the left square bracket. For example, to $t(x)=u[w_1\ldots w_{i-1}x\,w_{i+1}\ldots w_n]$ we associate the polynomial $t^{\overline{a}}(x)=u[w_1[\overline{a}]\ldots w_{i-1}[\overline{a}]\,x\,w_{i+1}[\overline{a}]\ldots w_n[\overline{a}]]$, to $t_1(x)=u[\overline{w}|_i\,v[\overline{s}|_jx]]$ we associate
$t^{\overline{a}}_1(x)=u[\overline{w}*\overline{a}|_i\,v[\overline{s}*\overline{a}|_jx]]$.

Using \eqref{e1} it is easy to show that
$$
t^{\overline{a}}(g[\overline{a}])=t(g)[\overline{a}]
$$
for all $g\in G,$ $\overline{a}\in B,$ $t\in T_n(G)$.

\begin{theorem}\label{T4.8}
Let $H$ be a nonempty bistrong subset of a Menger algebra $(G,o)$.
Then each $\mathcal{P}_H$-class $X$ different from $W^H$ is
bistrong, $W^H \subseteq W^X,$
$\mathcal{P}_H\subseteq\mathcal{P}_X$. Furthermore, the
restrictions of $\mathcal{P}_H$ and $\mathcal{P}_X$ on
$G\!\setminus W^X$ coincide.
\end{theorem}
\begin{proof}
Let $\sigma_{\!_X}\langle a\rangle\cap\sigma_{\!_X}\langle b\rangle\neq\varnothing$, then
$(\overline{x},t)\in\sigma_{\!_X}\langle a\rangle\cap\sigma_{\!_X}\langle b\rangle$ for some $t\in T_n(G)$ and $\overline{x}\in B.$ Therefore, $t(a[\overline{x}])\in X$ and
$t(b[\overline{x}])\in X.$ Since the $\mathcal{P}_H$-class $X$ is different from
$W^H$, by Proposition \ref{P4.7}, we have $X=\sigma^{\circ}_{_H}\langle\overline{y},t_1\rangle$ for some
$t_1\in T_n(G)$, $\overline{y}\in B.$ Consequently, $t(a[\overline{x}])\in\sigma^{\circ}_{_H}\langle\overline{y},t_1\rangle$ and
$t(b[\overline{x}])\in\sigma^{\circ}_{_H}\langle\overline{y},t_1\rangle$, whence $t_1(t(a[\overline{x}])[\overline{y}])\in H$ and
$t_1(t(b[\overline{x}])[\overline{y}])\in H.$ Thus,
$t_1(t^{\overline{y}}(a[\overline{x}][\overline{y}]))\in H$ and
$t_1(t^{\overline{y}}(b[\overline{x}][\overline{y}]))\in H,$ i.e.,
$t_1(t^{\overline{y}}(a[\overline{x}*\overline{y}]))\in H$ and
$t_1(t^{\overline{y}}(b[\overline{x}*\overline{y}]))\in H.$ This means that
$(t_1\circ t^{\overline{y}})(a[\overline{x}*\overline{y}])\in H$ and $(t_1\circ
t^{\overline{y}})(b[\overline{x}*\overline{y}])\in H,$ so
$(\overline{x}*\overline{y},t_1\circ t^{\overline{y}})\in\sigma_{\!_H}\langle
a\rangle$ and $(\overline{x}*\overline{y},t_1\circ t^{\overline{y}})\in\sigma_{\!_H}\langle b\rangle$. Thus,
$\sigma_{\!_H}\langle a\rangle\cap\sigma_{\!_H}\langle b\rangle\neq\varnothing$. Since $H$ is bistrong, the last implies $\sigma_{\!_H}\langle a\rangle=\sigma_{\!_H}\langle b\rangle$.
Now let $(\overline{z},t_2)\in\sigma_{\!_X}\langle a\rangle$, then $t_2(a[\overline{z}])\in X= \sigma^{\circ}_{_H}\langle\overline{y},t_1\rangle$, which gives
$t_1(t_2(a[\overline{z}])[\overline{y}])\in H.$ Consequently, $(t_1\circ t_2^{\overline{y}})(a[\overline{z} *\overline{y}])\in H,$ whence we obtain
$(\overline{z}*\overline{y},t_1\circ t_2^{\overline{y}})\in\sigma_{\!_H}\langle a\rangle=\sigma_{\!_H}\langle  b\rangle$. So, $(t_1\circ t_2^{\overline{y}})(b[\overline{z}*\overline{y}])\in H,$ i.e,
$t_1(t_2(b[\overline{z}])[\overline{y}])\in H.$ Hence, $t_2(b[\overline{z}])\in\sigma^{\circ}_{_H}\langle\overline{y},t_1\rangle=X,$ which means that $(\overline{z},t_2)\in\sigma_{\!_X}\langle b\rangle$. In this way we have proved the inclusion $\sigma_{\!_X}\langle a\rangle\subseteq\sigma_{\!_X}\langle b\rangle$. The proof of $\sigma_{\!_X}\langle
b\rangle\subseteq\sigma_{\!_X}\langle a\rangle$ is similar. Therefore, $\sigma_{\!_X}\langle a\rangle= \sigma_{\!_X}\langle b\rangle$. This proves that $X$ is a bistrong subset of $(G,o)$.

Let $(a,b)\in\mathcal{P}_H$, i.e., $\sigma_{\!_H}\langle a\rangle=\sigma_{\!_H}\langle b\rangle$. If
$(\overline{v},t_3)\in\sigma_{\!_X}\langle a\rangle$ for some $t_3\in T_n(G)$ and $\overline{v}\in B,$ then $t_3(a[\overline{v}])\in X.$ So, $t_3(a[\overline{v}])\in\sigma^{\circ}_{_H}\langle\overline{y},t_1\rangle$,
whence we obtain $t_1(t_3(a[\overline{v}])[\overline{y}])\in H.$ Thus, $t_1(t_3^{\overline{y}}(a[\overline{v}] [\overline{y}]))\in H,$ which gives $(t_1\circ t_3^{\overline{y}})(a[\overline{v}*\overline{y}])\in H.$ So, $(\overline{v}*\overline{y},t_1\circ t_3^{\overline{y}})\in\sigma_{\!_H}\langle a\rangle=\sigma_{\!_H}\langle  b\rangle$. Consequently, $(\overline{v},t_3)\in\sigma_{\!_X}\langle b\rangle$ and
$\sigma_{\!_X}\langle a\rangle\subseteq\sigma_{\!_X}\langle b\rangle$. Analogously we show the reverse inclusion. Therefore, $\sigma_{\!_X}\langle a\rangle=\sigma_{\!_X}\langle b\rangle$ and $(a,b)\in\mathcal{P}_X$. Hence $\mathcal{P}_H\subseteq\mathcal{P}_X$.

Thus, $W^H$ is contained in some $\mathcal{P}_X$-class. Let $a\in W^H,$ i.e.,
$\sigma_{\!_H}\langle a\rangle=\varnothing$. Suppose that
$\sigma_{\!_X}\langle a\rangle\neq\varnothing$. Then $t_4(a[\overline{w}])\in X$ for some $t_4\in T_n(G)$ and $\overline{w}\in B.$ Since $X=\sigma^{\circ}_{_H}\langle\overline{y},t_1\rangle$, the above gives
$t_4(a[\overline{w}])\in\sigma^{\circ}_{_H}\langle\overline{y},t_1\rangle$, so $t_1(t_4(a[\overline{w}]) [\overline{y}])\in H,$ therefore $(\overline{w}*\overline{y},t_1\circ t_4^{\overline{y}})\in\sigma_{\!_H}\langle a\rangle$, whence $\sigma_{\!_H}\langle a\rangle\neq\varnothing$, i.e., $a\not\in W^H,$ which is impossible, since $a\in W^H.$ Thus, our supposition is wrong. Therefore $\sigma_{\!_X}\langle a\rangle=\varnothing$, i.e., $a\in W^X.$ Hence $W^H\subseteq W^X.$

Since $\mathcal{P}_H \subseteq\mathcal{P}_X$, the set $W^X$ is a
union of $\mathcal{P}_H$-classes. We show that the restrictions of
$\mathcal{P}_H$ and $\mathcal{P}_X$ on $G\!\setminus W^X$
coincide. For this enough to show that for any $a,b\in
G\!\setminus W^X$ such that $(a,b)\in\mathcal{P}_X$ we have
$(a,b)\in\mathcal{P}_H$. In fact, if $a,b\in G\!\setminus W^X$ and
$(a,b)\in\mathcal{P}_X$, then $\sigma_{\!_X}\langle
a\rangle=\sigma_{\!_X}\langle b\rangle\neq\varnothing$. Thus,
there are $t_5\in T_n(G)$ and $\overline{u}\in B$ such that
$t_5(a[\overline{u}])\in
X=\sigma^{\circ}_{_H}\langle\overline{y},t_1\rangle$ and
$t_5(b[\overline{u}])\in
X=\sigma^{\circ}_{_H}\langle\overline{y},t_1\rangle$. Hence
$t_1(t_5(a[\overline{u}]) [\overline{y}])\in H$ and
$t_1(t_5(b[\overline{u}])[\overline{y}])\in H,$ so,
$(\overline{u}*\overline{y}, t_1\circ
t_5^{\overline{y}})\in\sigma_{\!_H}\langle a\rangle$ and
$(\overline{u}*\overline{y},t_1\circ
t_5^{\overline{y}})\in\sigma_{\!_H}\langle b\rangle$.
Consequently, $\sigma_{\!_H}\langle
a\rangle\cap\sigma_{\!_H}\langle b\rangle\neq\varnothing$, which
implies $\sigma_{\!_H}\langle a\rangle=\sigma_{\!_H}\langle
b\rangle$. Therefore, $(a,b)\in\mathcal{P}_H$. This completes our
proof.
\end{proof}

\begin{collolary}\label{C4.9}
Let $H$ be a bistrong subset of a Menger algebra $(G,o)$. Then for any translation $t\in T_n(G)$ and every
$\overline{x}\in B$ the set $\sigma^{\circ}_{\!_H}\langle\overline{x},t\rangle$ is bistrong.
\end{collolary}
\begin{proof}
It was noted above that the empty subset is bistrong, so if $\sigma^{\circ}_{\!_H}\langle\overline{x},t\rangle=\varnothing$ for some $t\in T_n(G)$ and $\overline{x}\in B,$ then, obviously, $\sigma^{\circ}_{\!_H}\langle\overline{x},t\rangle$ is bistrong. If $\sigma^{\circ}_{\!_H}\langle\overline{x},t\rangle\neq\varnothing$, then, by Proposition \ref{P4.7}, $\sigma^{\circ}_{\!_H}\langle\overline{x},t\rangle$ is a $\mathcal{P}_H$-class different from
$W^H.$ Therefore, by Theorem \ref{T4.8}, $\sigma^{\circ}_{\!_H}\langle\overline{x},t\rangle$ is bistrong.
\end{proof}

\begin{proposition}\label{P4.10}
Each $\varepsilon$-class $X$ of an $lv$-cancellative congruence $\varepsilon$ on a Menger algebra $(G,o)$ is  bistrong and $\varepsilon\subseteq\mathcal{P}_X$. Moreover, the relations $\varepsilon$ and $\mathcal{P}_X$ coincide on $G\!\setminus W^X.$
\end{proposition}
\begin{proof}
Let $\sigma_{\!_X}\langle a\rangle\cap\sigma_{\!_X}\langle b\rangle\neq\varnothing$, then
$(\overline{x},t)\in\sigma_{\!_X}\langle a\rangle$ and $(\overline{x},t)\in\sigma_{\!_X}\langle b\rangle$ for some $t\in T_n(G)$ and $\overline{x}\in B.$ Thus, $t(a[\overline{x}])\in X$ and $t(b[\overline{x}])\in X$, whence $(t(a[\overline{x}]),t(b[\overline{x}]))\in\varepsilon$. Since $\varepsilon$ is a $v$-cancellative congruence, the last implies $(a[\overline{x}],b[\overline{x}])\in\varepsilon$, which, by $l$-cancellativity, gives $(a,b)\in\varepsilon$. But $\varepsilon\subseteq\mathcal{P}_X$ (Proposition \ref{P4.2}), so, $(a,b)\in\mathcal{P}_X$, i.e., $\sigma_{\!_X}\langle a\rangle=\sigma_{\!_X}\langle b\rangle$. This shows that $X$ is a bistrong subset of $(G,o)$.

Moreover, from $\varepsilon\subseteq\mathcal{P}_X$ we obtain
$$
\varepsilon\cap(G\!\setminus\!W^X)\times(G\!\setminus\! W^X)\subseteq\mathcal{P}_X\cap(G\!\setminus\! W^X)\times (G\!\setminus\!W^X).
$$
On the other hand, if $(a,b)\in\mathcal{P}_X\cap(G\!\setminus\!W^X)\times (G\!\setminus\!W^X)$, then $a,b\in G\!\setminus\!W^X$ and
$$
t(a[\overline{x}])\in X\longleftrightarrow t(b[\overline{x}])\in X
$$
for all $t\in T_n(G)$, $\overline{x}\in B.$ Since for any $a\in G\!\setminus\!W^X$
there are $t_1\in T_n(G)$ and $\overline{y}\in B$ such that $t_1(a[\overline{y}])\in X,$ we have
$t_1(b[\overline{y}])\in X.$ Therefore, $(t_1(a[\overline{y}]),t_1(b[\overline{y}]))\in\varepsilon$,
whence $(a,b)\in\varepsilon$ because the relation $\varepsilon $ is $lv$-cancellative. So,
$$
\mathcal{P}_X\cap(G\!\setminus\!W^X)\times (G\!\setminus\!W^X)
\subseteq\varepsilon\cap(G\!\setminus\!W^X)\times(G\!\setminus\!W^X).
$$
From these two inclusions we obtain
$$
\varepsilon\cap(G\!\setminus~\!W^X)\times(G\!\setminus\!W^X)=\mathcal{P}_X\cap(G\!\setminus\!W^X)\times (G\!\setminus\!W^X).
$$
Thus, the relations $\varepsilon$ and $\mathcal{P}_X$ coincide on the set $G\!\setminus\!W^X$.
\end{proof}

\begin{minipage}{70mm}
\begin{flushleft}
Dudek~W. A. \\
 Institute of Mathematics and Computer Science \\
 Wroclaw University of Technology \\
 50-370 Wroclaw \\
 Poland \\
 Email: Wieslaw.Dudek@im.pwr.wroc.pl
\end{flushleft}
\end{minipage}
\hfill
\begin{minipage} {60mm}
\begin{flushleft}
 Trokhimenko~V. S. \\
 Department of Mathematics \\
 Pedagogical University \\
 21100 Vinnitsa \\
 Ukraine \\
 Email: vtrokhim@gmail.com
 \end{flushleft}
 \end{minipage}

\end{document}